\documentclass{amsart}
\usepackage{graphicx}
\usepackage{geometry}  
\usepackage{amssymb}
\usepackage[all]{xy}
\usepackage{xcolor}
\usepackage[hidelinks]{hyperref}
\usepackage{cleveref}
\usepackage{bm}
\numberwithin{equation}{section}

\newtheorem{theorem}[subsubsection]{Theorem}
\newtheorem{proposition}[subsubsection]{Proposition}

\newtheorem{corollary}[subsubsection]{Corollary}
\newtheorem{lemma}[subsubsection]{Lemma}
\newtheorem{conjecture}[subsubsection]{Conjecture}
\newtheorem*{theorem*}{Theorem}

\newtheorem{atheorem}{Theorem}

\newtheorem{aconjecture}[atheorem]{Conjecture}

\theoremstyle{definition}
\newtheorem{definition}[subsubsection]{Definition}

\theoremstyle{remark}
\newtheorem{example}[subsubsection]{Example}
\newtheorem{remark}[subsubsection]{Remark}
\newtheorem{question}[subsubsection]{Question}

\newcommand{\cR}{\mathcal{R}}
\newcommand\id{\operatorname{id}}
\newcommand\chara{\operatorname{char}}

\newcommand{\Rep}{\operatorname{Rep}}
\newcommand{\Res}{\operatorname{Res}}
\newcommand{\Ind}{\operatorname{Ind}}
\newcommand{\Hom}{\operatorname{Hom}}
\newcommand{\End}{\operatorname{End}}
\newcommand{\Frob}{\operatorname{Frob}}
\newcommand{\Fr}{\operatorname{Fr}}
\newcommand{\Stab}{\operatorname{Stab}}

\newcommand{\Fun}{\operatorname{Fun}}
\newcommand{\FPdim}{\operatorname{FPdim}}
\newcommand{\coker}{\operatorname{coker}}

\newcommand{\AbCat}{\mathcal{A}b\mathcal{C}at}
\newcommand{\Comod}{\mathrm{Comod}}

\newcommand{\Perm}{\mathrm{Perm}}
\newcommand{\Ver}{\operatorname{Ver}}
\newcommand{\Tilt}{\operatorname{Tilt}}
\newcommand{\Vecc}{\mathrm{Vec}}
\newcommand{\sVec}{\mathrm{sVec}}
\newcommand{\op}{\mathrm{op}}
\newcommand{\Tak}{\mathcal{T}ak}
\newcommand{\Tens}{\mathcal{T}ens}
\newcommand{\CoAl}{\mathcal{C}o\mathcal{A}l}

\newcommand{\mN}{\mathbb{N}}
\newcommand{\mZ}{\mathbb{Z}}
\newcommand{\mA}{\mathbb{A}}
\newcommand{\mF}{\mathbb{F}}

\newcommand{\mG}{\mathbb{G}}
\newcommand{\mP}{\mathbb{P}}
\newcommand{\mI}{\mathbb{I}}
\newcommand{\mJ}{\mathbb{J}}

\newcommand{\cA}{\mathcal{A}}
\newcommand{\cN}{\mathcal{N}}
\newcommand{\cB}{\mathcal{B}}
\newcommand{\cC}{\mathcal{C}}
\newcommand{\cV}{\mathcal{V}}
\newcommand{\cD}{\mathcal{D}}
\newcommand{\cP}{\mathcal{P}}

\newcommand{\cF}{\mathcal{F}}
\newcommand{\cE}{\mathcal{E}}
\newcommand{\cI}{\mathcal{I}}
\newcommand{\cJ}{\mathcal{J}}
\newcommand{\cK}{\mathcal{K}}

\newcommand{\bk}{\Bbbk}
\newcommand{\bV}{\mathbf{V}}
\newcommand{\bO}{\mathbf{O}}

\newcommand{\unit}{\mathbf{1}}
\newcommand\one{\mathbf{1}}
\newcommand\kk{\Bbbk}
\renewcommand\o{\otimes}
\newcommand{\tto}{\twoheadrightarrow}

\newcommand{\ulambda}{{\underline{\lambda}}}

\title[Higher Frobenius functors]{Towards higher Frobenius functors for symmetric tensor categories}
\author{Kevin Coulembier}
\address{School of Mathematics and Statistics, University of Sydney, Australia}
\email{kevin.coulembier@sydney.edu.au}
\author{Johannes Flake}
\address{Mathematical Institute, University of Bonn, Germany}
\email{flake@math.uni-bonn.de}

\begin{document}

\begin{abstract}
We develop  theory and examples of monoidal functors on tensor categories in positive characteristic that generalise the Frobenius functor from \cite{Os, EOf, Tann}. The latter has proved to be a powerful tool in the ongoing classification of tensor categories of moderate growth, and we demonstrate the similar potential of the generalisations. More explicitly, we describe a new construction of the generalised Verlinde categories $\Ver_{p^n}$ in terms of representation categories of elementary abelian $p$-groups. This leads to families of functors relating to $\Ver_{p^n}$ that we conjecture, and partially show, to exhibit the characteristic properties of the Frobenius functor relating to $\Ver_p$. In particular, we conjecture some of these functors to detect categories that fibre over $\Ver_{p^n}$.
\end{abstract}

\maketitle

\section{Introduction}

The structural study of (symmetric) tensor categories over a field, in the sense of \cite{DM, EGNO}, typically splits into categories of `moderate growth' and those of `superexponential growth', see \cite{CEO,HS} and references therein for recent progress in both directions. In the current paper we focus on the former type. In characteristic zero, by classical results of Deligne in \cite{Del90, Del02},
all such categories are super-tannakian, meaning they are representation categories of affine group superschemes. By the principle of tannakian reconstruction \cite{Del90} this is equivalent to saying that all tensor categories of moderate growth admit a tensor functor to (`fibre over') the tensor category of supervector spaces $\sVec$.

For the rest of the introduction, we fix a prime $p$, and work over an algebraically closed field $\bk$ of characteristic $p$. Then the situation is much more complicated, and we refer to \cite{BE, BEO, Tann, AbEnv, CEO, CEO2, EG, EOf, Os} and references therein for recent progress. We list the results that feature most prominently in the motivation for the current work. The symmetric fusion category $\Ver_p$, see \cite{Os}, which can be defined as the semisimplification of the representation category $\Rep C_p$ of the cyclic group of order $p$, is an example of a tensor category of moderate growth that is not super-tannakian (when $p>3$). In \cite{CEO}, generalising specific cases in \cite{EOf, Os}, it was proved that a tensor category of moderate growth fibres over $\Ver_p$ if and only it is `Frobenius exact'. The latter condition has many equivalent definitions, see \cite{Tann, CoNew, EOf}. The definition responsible for the terminology is that a tensor category $\cC$ is Frobenius exact when the Frobenius functor
\begin{equation}\label{FrIntro}\Fr:\;\cC\;\to\;\cC\boxtimes\Ver_p\end{equation}
is exact. This symmetric monoidal functor $\Fr$ is obtained from the $p$-th power functor $X\mapsto X^{\otimes p}$ on $\cC$, where we can interpret $X^{\otimes p}$ as a representation of $S_p$ or $C_p<S_p$, combined with the defining semisimplification functor $\Rep C_p\to \Ver_p$.

On the other hand, in \cite{BE, BEO, AbEnv} a chain of (incompressible) tensor categories
$$\Vecc\;\subset\;\sVec\;\subset\Ver_p\;\subset\;\Ver_{p^2}\;\subset\; \Ver_{p^3}\;\subset\;\cdots $$
is constructed where $\Ver_{p^n}$ is not Frobenius exact for $n>1$. In \cite[Conjecture~1.4]{BEO} it is conjectured that any tensor category of moderate growth fibres over $\Ver_{p^\infty}=\cup_n\Ver_{p^n}$. By \cite{CEO2}, this conjecture is equivalent to the conjecture that the only incompressible tensor categories of moderate growth are the tensor subcategories of $\Ver_{p^\infty}$.

\subsection{General theory of \texorpdfstring{$\bO$}{O}-functors} The aim of this paper, only partially fulfilled thus far, is to construct `higher analogues'
\begin{equation}\label{eqIntro}
    \cC\;\to\; \cC\boxtimes\Ver_{p^n}
\end{equation}
of $\Fr$ based on the functor $X\mapsto X^{\otimes p^n}$, for general $n\ge1$.
We axiomatise the essential properties of the Frobenius functor \eqref{FrIntro} yielding the notion of $\bO_n$-functors. In general terms, an $\bO_n$-functor is a family of symmetric monoidal functors $\Phi_\cC$ as in \Cref{eqIntro}, for each tensor category $\cC$ over $\bk$, which intertwines tensor functors. We say that an $\bO_n$-functor \emph{detects fibering over $\Ver_{p^n}$} if, for any tensor category $\cC$ of moderate growth, $\cC$ fibres over $\Ver_{p^n}$ if and only if $\Phi_\cC$ is exact. The main theorem of \cite{CEO} can then be rephrased as stating that $\Fr$ detects fibering over $\Ver_p$. Higher versions of the Frobenius functors should be $\bO_n$-functors that detect fibering over $\Ver_{p^n}$.

Using a distinguished generating object $V$ of the categories $\Ver_{p^n}$, a parity involution $X\mapsto\overline{X}$ on them, and a canonical semisimplification functor to $\Ver_p$, which allows us to associate to any $\bO_n$-functor $\Phi$ an $\bO_1$-functor $\overline{\Phi}$, we show in \Cref{ThmO}:

\begin{atheorem}
An $\bO_n$-functor $\Phi=(\Phi_\cC)_\cC$ detects fibering over $\Ver_{p^n}$ if
\begin{itemize}
\item $\Phi_{\Ver_{p^n}}(V)\simeq\overline{\one}\boxtimes\overline{V}$ and
\item the $\bO_1$-functor $\overline\Phi$ detects fibering over $\Ver_p$.
\end{itemize}
\end{atheorem}

We also show that, without any additional assumptions, any $\bO_n$-functor `defines' some incompressible tensor category, in a way such that the $\bO_1$-functor $\Fr$ defines $\Ver_p$. This gives further motivation for the construction of $\bO_n$-functors, although we also show that different $\bO_n$-functors can define the same incompressible category.

\subsection{Construction of \texorpdfstring{$\bO$}{O}-functors} To obtain concrete instances of potential higher Frobenius functors, we establish a new connection between the representation theory of $SL_2$ and that of the elementary abelian $p$-groups $C_p^n$. This is motivated by the fact that the categories $\Ver_{p^n}$ are defined in terms of $SL_2$. Only the special case $\Ver_p$ has thus far been defined via $C_p$ as well. Concretely, $\Ver_p$ is both the semisimplification of  $\Tilt SL_2$ and of $\Rep C_p$, and the latter definition is the one needed for \Cref{FrIntro}. 

Choose a faithful two-dimensional $C_p^n$-representation $V$ (in other words, an embedding of $C_p^n$ into $SL_2$) and let $\cD$ be the additive monoidal subcategory it generates. We identify an explicit tensor ideal $\cK$ in $\cD$, and show in \Cref{FirstResult} (see also \Cref{RemFirst}(1)):

\begin{atheorem}
$\Ver_{p^n}$ is the monoidal abelian envelope of the quotient category $\cD/\cK$.
\end{atheorem}

This generalises the definition of $\Ver_p$ used in the construction of $\Fr$ in \Cref{FrIntro}, and relates higher Verlinde categories to modular representations of elementary abelian $p$-groups, too. Indeed, in case $n=1$, we find $\cD=\Rep C_p$ and $\cK$ is the tensor ideal of negligible morphisms, so that $\cD/\cK=\Ver_p$.

Using an embedding $C_p^n<S_{p^n}$ this then allows us to define, in \Cref{Sec54}, functors 
$$\Phi^V_\cC:\cC\to\cC\boxtimes\Ver_{p^n}$$
generalising the construction of the Frobenius functor in \Cref{FrIntro} for $n=1$. 

Based on the known cases for $n=1$, which corresponds to the main theorem of \cite{CEO}, and our partial results for the case $p^n=4$ (as explained below), we formulate the following conjecture (the first part is \Cref{Conj}(3), the second part for $p^n=4$ is \Cref{ConjMG}):

\begin{aconjecture} \label{conj:higher-Frobenius}
The functors $\Phi^V_\cC$ define an $\bO_n$-functor. For certain natural choices of $V$, it detects fibering over $\Ver_{p^n}$. 
\end{aconjecture}

\subsection{Connection with modular representations of elementary \texorpdfstring{$p$
}{p}-groups}

We show that the part of \Cref{conj:higher-Frobenius} predicting that the family of functors $\Phi^V_\cC$ define an $\bO_n$-functor would follow from an intriguing conjecture about elementary abelian $p$-groups that can be stated without any reference to tensor categories. Let $V\in \Rep C_p^n$ again be any faithful two-dimensional representation. We prove that the above monoidal subcategory $\cD$ of $\Rep C_p^n$ has finitely many indecomposables up to isomorphism ({\it i.e.}~$V$ is algebraic) and that it has a minimal non-empty thick tensor ideal $I\subset\cD$.

Based on the case $n=1$, where this can been seen immediately, the case $p^n=4$, where it follows from the known classification of indecomposable modules for $C_2^2$, and some evidence in the general case that we derive in forthcoming work, we formulate a second conjecture in \Cref{Conj}(1):

\begin{aconjecture} \label{conj:T}
The full subcategory $I\subset \cD\subset\Rep C_p^n$ remains a thick tensor ideal when considered in $\Rep C_p^n$.
\end{aconjecture}

We then show in \Cref{ThmConj}:

\begin{atheorem} \label{prop:T-implies-O}
    \Cref{conj:T} implies that the $\Phi^V_\cC$ define an $\bO_n$-functor (i.e., the first part of \Cref{conj:higher-Frobenius}).
\end{atheorem}

\subsection{The case \texorpdfstring{$p^n=4$}{pⁿ=4}} In this case we can prove \Cref{conj:T} and partially also \Cref{conj:higher-Frobenius}.
From \Cref{prop:T-implies-O}, we obtain $\bO_2$-functors
\begin{equation}\label{eqIntro2}
    \cC\;\to\;\cC\boxtimes\Ver_{4}.
\end{equation}
We prove that, for two (Galois conjugate) choices of $V$, the functor in \Cref{eqIntro2} is exact for $\cC=\Ver_4$ and hence is exact on every tensor category that fibres over $\Ver_4$. We make some progress towards proving that, for these choices of $V$, the functor in \Cref{eqIntro2} being exact on $\cC$ is also a sufficient condition for $\cC$ to fibre over $\Ver_4$. In particular, we show in \Cref{ThmVer4}:

\begin{atheorem}
Assuming \cite[Conjecture~1.4]{BEO}, the constructed $\bO_2$-functor detects fibering over $\Ver_4$ at least for finite tensor categories.
\end{atheorem}

In this case we also derive a large family of new $\bO_1$-functors based on the fourth tensor power. Again under \cite[Conjecture~1.4]{BEO} these should, at least on finite tensor categories, all single out (via the exactness condition) tannakian categories. Our analysis also leads to unconditional results, for instance the following result in \Cref{CorVer2n}:

\begin{atheorem} Consider $n>2$. For any two projective objects $P,Q$ in $\Ver_{2^n}$ or $\Ver_{2^n}^+$, the space $\Hom(P,Q^{\o4})$, viewed as a representation over the Klein 4-subgroup $C_2^2<S_4$, is a direct sum of non-trivial permutation modules.
\end{atheorem}

\subsection*{Structure of the paper} In \Cref{SecPrel} we recall some required background. In \Cref{SecDTP} we explain some properties of the Deligne tensor product that are known but not always easy to locate in the literature. In \Cref{SecO} we introduce and develop the theory of $\bO$-functors. In \Cref{SecConst} we introduce a conjectural $\bO$-functor corresponding to each $\Ver_{p^n}$. In \Cref{SecKlein} we classify linear symmetric monoidal functors from $\Rep_{\bk} C_2^2$, for $\mathrm{char}(\bk)=2$, to tensor categories, leading to a large family of $\bO$-functors, including the main `conjectural one' from \Cref{SecConst} for $\Ver_4$. In \Cref{FinSec} we describe which tensor categories the above $\bO$-functors single out via their exactness condition.

\section{Preliminaries}\label{SecPrel}

Throughout, $\bk$ will denote an algebraically closed field, usually of characteristic $p>0$. Then we have the Frobenius automorphism
$$\Frob:\,\bk\to\bk\,,\; \lambda\mapsto \lambda^p.$$
We consider the symmetric group $S_n$ as the permutation group of the set 
$\{1,2,\dots, n\}.$ 

We set $\mN:=\{0,1,\dots\}$.

\subsection{Tensor categories and functors}
We refer to \cite{EGNO} for basic notions regarding monoidal categories.

\subsubsection{} We will abbreviate `$\bk$-linear symmetric monoidal category' to $\bk$-LSM category, leaving out $\bk$ when it is clear from context. It is understood that the tensor product in such a category is $\bk$-linear in each variable. Similarly, we speak of LSM functors between LSM categories. 

More generally, for $\bk_i$-LSM categories $\cC_i$, $i\in\{1,2\}$ and a field homomorphism $\varphi:\bk_1\to\bk_2$, a $\varphi$-LSM functor $F:\cC_1\to\cC_2$ is a symmetric monoidal functor $F$ such that 
$$F(\lambda f)=\varphi(\lambda)F(f),\quad\mbox{for all $\lambda\in\bk_1$},$$
and all morphisms $f$ in $\cC_1$.

Usually, this situation will occur for $\bk_1=\bk_2=\bk$ and $\varphi$ a power of $\Frob$.

\subsubsection{} A $\bk$-LSM category
$(\cC,\otimes,\unit)$ is a {\bf tensor category over $\bk$} if
\begin{enumerate}
\item $\cC$ is abelian with objects of finite length;
\item $\bk\to\End_{\cC}(\unit)$ is an isomorphism;
\item $(\cC,\otimes,\unit)$ is rigid, meaning that every object $X$ has a monoidal dual $X^\vee$.
\end{enumerate}
Such categories are also sometimes called symmetric tensor categories or pretannakian categories. Note that $\unit$ is a simple object, see \cite[Theorem~4.3.8]{EGNO}. It then follows from the assumptions (1)-(3) that morphism spaces in a tensor category are finite-dimensional.

The first example of a tensor category is the category $\Vecc=\Vecc_{\bk}$ of finite-dimensional vector spaces over $\bk$. We denote the category of all vector spaces by $\Vecc^\infty$.

For tensor categories, we will usually leave out the reference `over $\bk$' when it is clear from context over which field we work.

A {\bf $\varphi$-tensor functor} is an exact $\varphi$-LSM functor between tensor categories over the source and target of a field homomorphism $\varphi$. We leave out reference to $\varphi$ when $\varphi=\id$.

A {\bf tensor subcategory} of a tensor category is a full subcategory which is closed under taking subquotients, tensor products and duals. A tensor functor $F:\cC_1\to\cC_2$ is called {\bf surjective} if every object in $\cC_2$ is a subquotient of one in the essential image of $F$. In other words, there is no proper tensor subcategory of $\cC_2$ in which~$F$ takes values.

\begin{example}\label{prelim::representations}
\begin{enumerate}
    \item For an abstract group $H$, we have the tensor category $\Rep_{\bk}H=\Rep H$ of finite-dimensional representations.
    \item  For an affine group scheme $G$ over $\bk$, we have the tensor category $\Rep_{\bk}G=\Rep G$ of rational representations (comodules over the coordinate ring). In particular, we have a (forgetful) tensor functor 
$$
\Rep_\bk G \to \Rep_\bk G(\bk).
$$
\end{enumerate}
    
\end{example}

\subsubsection{} Following \cite[Chapter~6]{EGNO}, we say that a tensor category is {\bf finite}, if it has enough projective objects and only finitely many simple objects up to isomorphism. This is equivalent to demanding that the underlying abelian category is equivalent to the category of finite-dimensional modules of a finite-dimensional associative $\bk$-algebra.

Note also that in a tensor category $\cC$, the existence of one (non-zero) projective object implies that $\cC$ has enough projective objects; and projective objects are injective.

We record some weaker `finiteness' conditions on tensor categories:
\begin{enumerate}
    \item[(S)] $\cC$ is a union of tensor (sub)categories with finitely many simple objects, up to isomorphism;
    \item[(P)] $\cC$ is a union of tensor categories which have enough projective objects (or equivalently, which have non-zero projective objects);
    \item[(SP)] $\cC$ satisfies (S) and (P), or equivalently, it is a union of finite tensor categories.
\end{enumerate}
Note that these three properties are inherited by tensor subcategories. On tensor categories satisfying (S), we have a well-defined notion of Frobenius--Perron dimension $\FPdim$ by \cite[Proposition 4.5.7]{EGNO}.

An example of a tensor category satisfying (SP) is $\Ver_{p^\infty}$.

Another finiteness condition that we can impose extends the notion of algebraic representations due to Alperin. Concretely, we say that an object in a tensor category is {\bf algebraic} if there are only finitely many indecomposable summands (up to isomorphism) in its tensor powers. Equivalently, an object is algebraic if its image in the split Grothendieck ring satisfies a polynomial equation.

\subsubsection{}A tensor category $\cC$ is of {\bf moderate growth}, see \cite{BE, CEO, CEO2}, if for every $X\in\cC$, the limit (the \emph{growth dimension} of $X$)
$$\lim_{n\to\infty}\sqrt[n]{\ell(X^{\otimes n})}\;\in\;\mathbb{R}_{\ge 0}\cup\{\infty\}$$
is finite, see \cite[\S 4]{CEO}, where $\ell$ denotes the length of an object. 

It follows from the alternative characterisation in \cite[Lemma~4.11(3)]{CEO} that if $\cC$ is of moderate growth and there exists a surjective tensor functor $F:\cC\to\cC'$, then also $\cC'$ is of moderate growth. Tensor categories satisfying (S) are of moderate growth, since $\ell(X)\le\FPdim(X)$ and the Frobenius--Perron dimension is multiplicative.

\subsection{Exactness criteria}

We will be particularly interested in LSM functors between tensor categories that are not (necessarily) exact.

The following property, see \cite[Theorem~2.4.1]{CEOP} or \cite{Top}, will be used often.
\begin{theorem}\label{Thm:FE}
    Let $F:\cC_1\to\cC_2$ be a $\varphi$-LSM functor between tensor categories over fields $\bk_1,\bk_2$. Then $F$ is exact ({\it i.e.}~a $\varphi$-tensor functor) if and only if it is faithful.
\end{theorem}

\begin{lemma}\label{ExFin}\label{LemK0}
Consider a $\varphi$-LSM functor $F:\cC\to\cC'$ between tensor categories, where $\cC$ satisfies {\rm (P)}. The following conditions are equivalent:
    \begin{enumerate}
    \item $F$ is exact.
    \item $F$ is faithful.
    \item There are tensor subcategories $\{\cC_\alpha\subset\cC\}$ satisfying $\cup_\alpha\cC_\alpha=\cC$, with projective objects $P_\alpha\in\cC_\alpha$ such that $F(P_\alpha)\not=0$.
      
    \end{enumerate}
\end{lemma}
\begin{proof}That (1) and (2) are equivalent is \Cref{Thm:FE}.
Clearly (2) implies (3). 
    Since tensoring with a projective object splits every short exact sequence, it follows that (3) implies (1).
\end{proof}


\begin{corollary}\label{CorExFin}
    Consider a surjective $\varphi$-tensor functor $F:\cC\to\cC'$ between finite tensor categories and a $\varphi'$-LSM functor $F':\cC'\to\cC''$ to a third tensor category such that the composite $F'\circ F$
        is exact. Then $F'$ is a $\varphi'$-tensor functor.
\end{corollary}
\begin{proof}
    This follows from the equivalence between (1) and (2) in \Cref{ExFin} and the observation, see \cite[Theorem~6.1.16]{EGNO}, that $F$ sends projective objects in $\cC$ to projective objects in $\cC'$.
\end{proof}

\subsubsection{} 
As recalled in \Cref{Thm:FE}, an LSM functor between tensor categories is not exact if and only if its kernel is a non-trivial tensor ideal. In general for an LSM category $\cA$ and a tensor ideal $I$ in $\cA$ it will be understood that $\cA\to\cA/I$ stands for the LSM functor which sends every object to itself and every morphism to its equivalence class.   An important example, for any tensor category $\cC$, is the {\bf semisimplification}
$$\Sigma: \cC\to\overline{\cC},$$
see \cite{EOs} for an overview. Most importantly, $\overline{\cC}$ is the quotient with respect to the unique maximal tensor ideal and is again a tensor category, and moreover semisimple.

For a finite group $G$, the minimal non-zero tensor ideal in $\Rep_{\bk} G$ (say over a field~$\bk$ such that $\mathrm{char}(\bk)$ divides $|G|$) is the ideal of all morphisms that factor through projective objects. The resulting quotient is the stable category $\Stab G$.

\subsection{Incompressible categories} \label{SecInc}

An {\bf incompressible (tensor) category}, see \cite{CEO2} and references therein, is a tensor category $\cC$ for which every tensor functor $\cC\to\cC'$ realises an equivalence with a tensor subcategory of $\cC'$.

In this section, we assume that $\chara(\bk)=p>0$. The list of currently known incompressible tensor categories is given in \cite{BEO}. We review some essential properties and conjectures concerning these.

\subsubsection{} \label{TiltSL2}
Recall from \Cref{prelim::representations} that $\Rep_{\bk} SL_2$ is the category of finite-dimensional rational $SL_2(\bk)$-representations. 
The full subcategory of $\Rep_{\bk}SL_2$ of tilting modules, see \cite[Chapter~E]{Jantzen}, is denoted by $\Tilt SL_2$. This is also the full monoidal subcategory comprising direct summands of direct sums of tensor powers of the natural (2-dimensional) representation~$V$. Indecomposable tilting modules are labelled by their highest weight, and we denote them as $T_i$, $i\in\mN$. For instance, $T_0\simeq \unit$ and $T_1\simeq V$.

By \cite[Theorem~5.3.1]{Selecta}, the proper tensor ideals in $\Tilt SL_2$ form one countable chain 
$$I_1\supset I_2\supset I_3\supset \cdots,$$
with $I_n$ the ideal generated by the Steinberg module $St_n=T_{p^n-1}$. More precisely, $I_n$ comprises all morphisms that factor through a direct sum of $T_j$ with $j\ge p^n-1$. Furthermore,
$$I_n(T_i,T_j)\;=\;0$$
unless $i+j> 2p^n-2$. In particular, if we denote by $\Tilt^nSL_2$, for $n\in\mZ_{>0}$ the full subcategory of $\Tilt SL_2$ comprising direct sums of copies of $T_i,i<p^n-1$, then the composite
$$\Tilt^n SL_2\;\hookrightarrow\; \Tilt SL_2\;\to\; (\Tilt SL_2)/I_n$$
is fully faithful. We denote further by $\Tilt^{[n]}SL_2$, for $n\in\mZ_{>0}$ the full subcategory of $\Tilt SL_2$ comprising direct sums of copies of $T_i$ with $p^{n-1}-1\le i<p^n$.

\subsubsection{}\label{abenv} By \cite{BEO, AbEnv}, for each $n\in\mZ_{>0}$, the LSM category $(\Tilt SL_2)/I_n$ admits a `monoidal abelian envelope'. Concretely, for each $n\in\mZ_{>0}$ we have an LSM-functor
$$\Sigma^n:\;\Tilt SL_2\to\Ver_{p^n}$$
with kernel $I_n$, to a tensor category $\Ver_{p^n}$. Note that $\Sigma^1$ corresponds to the semisimplification of the (non-abelian) symmetric monoidal category $\Tilt SL_2$. We will use the same notation for the corresponding faithful LSM-functor
\begin{equation}\label{Sigman2}
    \Sigma^n:\;(\Tilt SL_2)/I_n\to\Ver_{p^n}.
\end{equation}
Moreover, $\Sigma^n$ restricts to an equivalence
\begin{equation}\label{EquivProj}\Tilt SL_2\supset\Tilt^{[n]}SL_2\;\xrightarrow{\sim}\; \cP\subset\Ver_{p^n}
\end{equation}
with the category $\cP$ of projective objects in $\Ver_{p^n}$, see \cite[Theorem~4.5]{BEO}.

It is also proved in \cite{BEO, AbEnv} that, for each tensor category $\cC$, restriction along $\Sigma^n$ in \Cref{Sigman2} gives an equivalence between the category of tensor functors $\Ver_{p^n}\to\cC$ and the category of faithful LSM functors $(\Tilt SL_2)/I_n\to \cC$.

\subsubsection{}It is proved in \cite{BEO} that each $\Ver_{p^n}$ is incompressible and that the simple objects in $\Ver_{p^n}$ are labelled by 
$$\{L_i^{[n]}\mid 0\le i\le p^{n-1}(p-1)-1\}$$
such that
$$L_0^{[n]}\simeq \unit,\quad L_1^{[n]}\simeq \Sigma^n(T_1)=\Sigma^n(V), \quad\mbox{and}\quad L^{[n]}_{p^{n-1}(p-2)}=\bar{\unit},$$
where the last example assumes $p>2$ and denotes by $\bar{\unit}$ the odd line in the category $\sVec$ of supervector spaces.
When no confusion is possible, we will write $L_i$ instead of $L_i^{[n]}$. However, we have inclusions $\Ver_{p^n}\subset \Ver_{p^{n+1}}$ which identify $L_i^{[n]}$ with~$L_{pi}^{[n+1]}$. 

We will also refer to the subcategories $\Ver_{p^n}^+\subset\Ver_{p^n}$ from \cite[\S 4.1]{BEO}. The tensor category $\Ver_{p^n}^+$ is a direct summand of $\Ver_{p^n}$ that contains half of the simple objects and corresponds to the $SL_2$-tilting modules with even highest weight under~\eqref{EquivProj}. Moreover, $\Ver_{p^n}\simeq \Ver_{p^n}^+\boxtimes\sVec$ if $p>2$. 
We will be particularly interested in $\Ver_4^+$ and $\Ver_4$, described as $\cC_1$ and $\cC_2$ in \cite{BE}.

The following theorem is due to the first author, Etingof and Ostrik, see \cite[Theorem~1.1]{CEO}. Special cases were obtained earlier in for instance \cite{EG, EOf, Os}. For the notion of Frobenius exactness we refer to \cite{CEO, EOf}, or \cite[Theorem~C]{Tann} or \Cref{ExOFirst}(2) and \Cref{ExCpFr} below.

\begin{theorem}[$\Ver_p$-Theorem]
    The following conditions are equivalent on a tensor category $\cC$:
    \begin{enumerate}
        \item $\cC$ is of moderate growth and Frobenius exact;
        \item $\cC$ admits a tensor functor to $\Ver_p$.
    \end{enumerate}
    Moreover, the tensor functor in (2) is unique up to isomorphism.
    \label{pTheorem}
\end{theorem}
This is to be compared with the following conjecture due to Benson, Etingof and Ostrik, see \cite[Conjecture~1.4]{BEO}.
\begin{conjecture}[$\Ver_{p^\infty}$-Conjecture]
    The following conditions are equivalent on a tensor category $\cC$:
    \begin{enumerate}
        \item $\cC$ is of moderate growth;
        \item $\cC$ admits a tensor functor to $\Ver_{p^\infty}$.
    \end{enumerate}
    
    \label{BEOConjecture}
\end{conjecture}

We conclude the preliminaries with some considerations regarding exactness of LMS functors out of our incompressible categories. We fix $n\in\mZ_{>0}$ and $i\in\mZ$. The Frobenius--Perron dimensions of all simple objects in $\Ver_{p^n}$ have been calculated in \cite[Corollary~4.44]{BEO}. We use the result freely in the following two proofs.

\begin{lemma} Assume $p^n>2$.
Consider a field homomorphism $\varphi:\bk\to\bk'$.
    A $\varphi$-LSM functor $F:\Ver_{p^n}\to\cC$ to a finite tensor category $\cC$ over $\bk'$ is exact if and only if
    $$\FPdim(F(L_1))\;=\; \FPdim(L_1)\;=\; 2\cos\left(\frac{\pi}{p^n}\right),$$
    which is the case if and only if $\FPdim(L_1)\not\in\{2\cos(\frac{\pi}{p^m}):m<n\}$.
\end{lemma}
\begin{proof}
In case $F$ is exact, it preserves Frobenius--Perron dimension as a consequence of \cite[Proposition~3.3.13]{EGNO}.
Consider the LSM functors
$$(\Tilt SL_2)/ I_n\;\xrightarrow{\Sigma^n}\;\Ver_{p^n}\;\xrightarrow{F}\;\cC.$$
Assume that $F$ is not exact. Then \Cref{ExFin} and the fact that all projective objects in $\Ver_{p^n}$ are in the essential image of $\Sigma^n$ imply that $F\circ\Sigma^n$ is not faithful. Hence, $F\circ\Sigma^n$ factors via $(\Tilt SL_2)/ I_m$ for $m<n$. We take the minimal such $m$. By the abelian envelope property of \Cref{abenv}, if follows that $F\circ\Sigma^n$ is also the composite of LSM functors
$$(\Tilt SL_2)/ I_n\;\tto\;(\Tilt SL_2)/ I_m\;\xrightarrow{\Sigma^m}\;\Ver_{p^m}\;\to\;\cC,$$
where the right-most functor is exact. By the first paragraph we find that 
$$\FPdim(F(L_1))\;=\; 2\cos\left(\frac{\pi}{p^m}\right)\;<\; 2\cos\left(\frac{\pi}{p^n}\right),$$
concluding the proof.
\end{proof}

We refer to the next section, and references therein, for a full treatment of the Deligne tensor product $\boxtimes$ appearing in the next corollary.

\begin{corollary}\label{CorVVV} Assume $p^n>2$.
    A $\Frob^i$-LSM functor $$F:\Ver_{p^n}\to\Ver_{p^n}\boxtimes \Ver_{p^n}$$ is exact if and only if we are in one of the following cases:
    \begin{enumerate}
        \item $F(L_1)=\unit\boxtimes L_1$, or $F(L_1)=\bar{\unit}\boxtimes(\bar{\unit}\otimes L_1)$ if $p>2$, yielding equivalences
        $$\Ver_{p^n}\xrightarrow{\sim}\Vecc\boxtimes \Ver_{p^n}\quad\mbox{or}\quad\Ver_{p^n}\xrightarrow{\sim}\sVec\boxtimes \Ver_{p^n}^+.$$
        \item $F(L_1)= L_1\boxtimes \unit$, or $F(L_1)=(\bar{\unit}\otimes L_1)\boxtimes \bar{\unit}$ if $p>2$, yielding similar equivalences.
    \end{enumerate}
\end{corollary}

\begin{remark}\label{RemFrobVer}
    There exists a $\Frob$-LSM auto-equivalence ({\it i.e.}~a $\Frob$-tensor auto-equivalence) of $\Ver_{p^n}$. Indeed, for a $\bk$-LSM category $\cA$, denote by $\cA^{(1)}$ the same pre-additive symmetric monoidal category, but with a new $\bk$-linear structure where the new action of $\lambda\in\bk$ on a morphism $f$ is given by $\lambda^p f$ in terms of the old linear structure. Then $\cA^{(1)}$ is again a $\bk$-LSM category, a priori different from~$\cA$. We can thus reformulate the claim as saying that $\Ver_{p^n}$ and $(\Ver_{p^n})^{(1)}$ are actually equivalent as tensor categories over $\bk$. Now $\Tilt SL_2$ and $(\Tilt SL_2)^{(1)}$ are equivalent, for instance because $SL_2$ can be defined over $\mF_p$, or by using the equivalence with the diagrammatic Temperley--Lieb category, and this equivalence exchanges the ideals $I_n$. We can then realise $\Ver_{p^n}$ and $(\Ver_{p^n})^{(1)}$ as the abelian envelopes of the equivalent categories $(\Tilt SL_2)/I_n$ and $((\Tilt SL_2)/I_n)^{(1)}$.
\end{remark}


\section{Deligne products}\label{SecDTP}
In this section, we do not claim originality. We write out some useful aspects of the Deligne tensor product as first introduced in \cite{Del90}. Many of these have appeared implicitly in \cite{Tann, CEO, EOf}, but not always in full detail or with complete proof. Recall that $\bk$ is an algebraically closed field.

\subsection{Comodules of coalgebras}

\subsubsection{}
For a coalgebra $C$ over $\bk$, we consider the categories of finite-dimensional, and all, left comodules, respectively:
$$\Comod_{C}\quad\mbox{and}\quad \Comod^\infty_{C}\simeq\Ind(\Comod_{C}).$$

 For a right $C$-comodule $M$ and a left $C$-comodule $N$ we have the cotensor product
$$M\Box^C N\;=\; \mathrm{Eq}(M\otimes_{\bk}N\rightrightarrows M\otimes_{\bk}C\otimes_{\bk}N),$$
which is a vector space.
The following lemma is well-known.
\begin{lemma}\label{LemComod}
    \begin{enumerate}
        \item A right $C$-comodule $M$ is coflat ($M\Box^C-$ is exact) if and only if it is injective.
        \item For coalgebras $C_1,C_2$, consider a right $C_1\otimes C_2$-comodule $T$ for which
    $$\Comod_{C_1}\times\Comod_{C_2}\,\to\, \Vecc^\infty,\;\, (M_1,M_2)\mapsto T\Box^{C_1\otimes C_2}(M_1\otimes M_2) $$
    is exact in both variables.
    Then $T$ is injective over $C_1\otimes C_2$.
        \item For coalgebras $C_1,C_2$, the assignment which sends a $(C_2,C_1)$-bicomodule $M$ to the functor
        $$M\Box^{C_1}-:\; \Comod_{C_1}\to\Comod_{C_2}^\infty$$
        yields an equivalence between the category of right injective bicomodules and the category of exact $\bk$-linear functors.
    \end{enumerate}
\end{lemma}
\begin{proof}
    For part (1) it suffices to observe that
    $$M\Box^C-\;\simeq\;\Hom_{C}(-^\ast, M)$$
    as functors from $\Comod_{C}$ to $\Vecc^\infty$, where we interpret the dual space of a finite-dimensional left $C$-comodule canonically as a right $C$-comodule.

    For part (2), it is sufficient to show that $T$ is coflat, by part (1). This is a straightforward generalisation of \cite[Proposition~5.7]{Del90}. We can express the coflatness as vanishing of a first derived functor; and it is sufficient to prove vanishing when evaluated on simple comodules, since the cohomology functors commute with direct limits and are exact in the middle.

    Now, the simple $B_1\otimes B_2$-comodules are precisely those of the form $V_1\otimes V_2$ for simple $B_i$-comodules $V_i$. This is true because $\bk$ is algebraically closed, see for instance \cite[Lemme~5.9]{Del90}. 
    
    However, vanishing of the cohomology functor on $B_1\otimes B_2$-comodules that are of the exterior tensor product form $V_1\otimes V_2$ follows from the assumptions.

    Part (3) is standard, using part (1).
\end{proof}

\subsection{The Deligne product of abelian categories}

Let $\cA$ and $\cB$ be two essentially small $\bk$-linear abelian categories.
\subsubsection{}\label{DefDel}  The Deligne tensor product $\cA\boxtimes\cB=\cA\boxtimes_{\bk}\cB$ (if it exists) is a $\bk$-linear abelian category, equipped with a bilinear bifunctor 
\begin{equation}
\label{DefboxDel}-\boxtimes-:\;\cA\times\cB\to\cA\boxtimes\cB, \quad (X,Y)\mapsto X\boxtimes Y,    
\end{equation}
 satisfying the following universal property.
For every $\bk$-linear abelian category $\cC$, restriction along \Cref{DefboxDel} yields an equivalence between the category of right exact $\bk$-linear functors $\cA\boxtimes \cB\to\cC$ with the category of bilinear bifunctors $\cA\times\cB\to\cC$ that are right exact in each variable. 

As demonstrated in \cite{LF}, the Deligne tensor product need not always exist. Of course, when $\cA\boxtimes\cB$ exists, it is unique up to equivalence. 

\subsubsection{}\label{DefDel2} Keeping the set-up of \Cref{DefDel}, assume now that all morphisms spaces in $\cA$ are finite-dimensional and all objects have finite length. It then follows from Takeuchi's theorem, see \cite{Tak} or \cite[Theorem~1.9.15]{EGNO}, that there exists a faithful exact $\bk$-linear functor $\omega:\cA\to\Vecc$, which allows us to identify $\cA$ with the category of finite-dimensional comodules of the coalgebra $$C\;=\;\int^{X\in\cA}\omega(X)^\ast\otimes_{\bk}\omega(X).$$


We define the (abelian) category $\Comod_C(\cB)$ of (left) $C$-comodules in $\cB$ as the category of objects $Y$ in $\cB$, equipped with an algebra morphism
$$C_Y^\ast\;\to\; \End_{\cB}(Y),$$
for some finite-dimensional subcoalgebra $C_Y\subset C$. Here we use the fact that a coalgebra is the union of its finite-dimensional subcoalgebras, allowing use to avoid using the ind-completion of $\cB$ in the definition. We have an obvious bilinear bifunctor
$$\cA\times\cB\;\to\; \Comod_C\cB.$$

The following lemma is a minor generalisation of \cite[Theorem~3]{LF} (in the sense that no finiteness assumptions are made on $\cB$), but more importantly includes the explicit realisation of the Deligne product that we need.
\begin{lemma}\label{LemExDTP}
    Under the assumption on $\cA$ in \Cref{DefDel2} and with the corresponding $C$ as defined there, we have:
    \begin{enumerate}
        \item The Deligne product $\cA\boxtimes\cB$ exists and is given by $\Comod_C\cB$.
        \item Every object in $\cA\boxtimes\cB$ is a quotient of one of the form~$X\boxtimes Y$.
        \item Every object in $\cA\boxtimes\cB$ is a subobject of one of the form~$X\boxtimes Y$.
    \end{enumerate}
\end{lemma}
\begin{proof}
    By the general theory of \cite{LF}, for instance Remark 10 and Theorem 18, we need to show that $\Comod_C\cB$ is equivalent to the category of compact (finitely presented) objects in the category of bilinear bifunctors
    $$\cA^{\op}\times\cB^{\op}\;\to\;\Vecc^\infty$$
    that are left exact in each variable.
    The latter category can be reinterpreted as the category of linear left exact functors from~$\cB^{\op}$ to 
    $$\Ind\cA\;\simeq\;\Comod^\infty_C,$$
    the category of all (not necessarily finite-dimensional) $C$-comodules over $\kk$. Using the forgetful functor from $\Comod^\infty_C$ to $\Vecc^\infty$, we can observe that we can reinterpret (left exact) functors from $\cB^\op$ to $\Comod^\infty_C$ as (left exact) functors from $\cB^\op$ to $\Vecc^\infty$ equipped with a $C$-comodule structure. In conclusion, we have identified our original functor category with the category
    of $C$-comodules in $\Ind\cB$ (in the ordinary sense, meaning objects $Y\in\Ind\cB$ with a coaction $Y\to C\otimes_{\bk}Y$). Its subcategory of compact objects is indeed $\Comod_C(\cB)$: To show that compact $C$-comodules in $\Ind\cB$ must have their underlying object in $\cB$, it suffices to test against comodules of the form $C\otimes_{\bk}(\mathrm{colim} Y_i)$  with $Y_i\in\cB$. To see that the underlying object being in $\cB$ is sufficient to be compact for a $C$-comodule in $\Ind\cB$ one can realise $C$-linear morphisms spaces as equalisers of  $\cB$-morphism spaces.
    This concludes the proof of part (1).

    For  $M$ in $\Comod_C\cB$ realised as $C_{M_0}^\ast\to \End_{\cB}(M_0)$, with $M_0\in\cB$,
    $$ M\;\hookrightarrow\; C_{M_0}\boxtimes M_0\quad\mbox{and}\quad C_{M_0}^\ast\boxtimes {M_0}\;\tto\; M$$
are an obvious monomorphism and epimorphism proving parts (2) and (3).
\end{proof}

\begin{example}\label{ExRepG}
\begin{enumerate}
    \item For a finite group $G$ and a $\bk$-linear abelian category $\cB$,
    $$\Comod_{(\bk G)^\ast}\cB\;\simeq\; (\Rep G)\boxtimes\cB\;\simeq\; \cB\boxtimes (\Rep G)\;\simeq\;\Fun(BG,\cB)$$
    is the category of $G$-representations in $\cB$.
    \item More generally, assume that $\cA$ is a finite category in the sense of \cite{EGNO}, meaning it is equivalent to the category of finite-dimensional modules over a finite-dimensional associative algebra $A$. Then we can simplify all the above constructions by replacing the (now finite-dimensional) coalgebra $C$ by its dual algebra $C^\ast=A\simeq \End(\omega)$ (for $\omega:\cA\to\Vecc$ as above). So $\cA\boxtimes\cB$ is the category of $A$-modules in $\cB$. 
\item  Yet equivalently, for $\cA$ as in (2), we have
$$\cA\boxtimes\cB\simeq\Fun_{\bk}(\cP^{\op},\cB),$$
with $\cP$ the category of projective objects in $\cA$. 
\end{enumerate}
    
\end{example}

\subsubsection{}\label{DefDel3} Denote by $\AbCat_{\bk}$ the 2-category whose objects are essentially small $\bk$-linear categories, 1-morphisms are $\bk$-linear (not necessarily exact) functors and 2-morphisms are all natural transformations.
We denote the full 2-subcategory $\Tak_{\bk}$ (after Takeuchi) of $\bk$-linear abelian categories with objects of finite length and morphism spaces that are finite-dimensional. 
We consider 2-subcategories
$$\AbCat^{ex}_{\bk}\subset\AbCat^{rex}_{\bk}\subset\AbCat_{\bk}\quad\mbox{and}\quad \Tak^{ex}_{\bk}\subset\Tak^{rex}_{\bk}\subset\Tak_{\bk}$$
by keeping the same objects but restricting the 1-morphisms to (right) exact functors.

\begin{proposition}\label{PropPseudo}
There is a pseudo-functor
    $$-\boxtimes-:\;\Tak_{\bk}^{ex}\times \AbCat_{\bk}\;\to\; \AbCat_{\bk}$$
    that agrees with the pseudo-functor (implicit in the definition of $\boxtimes$)
    $$-\boxtimes-:\;\Tak_{\bk}^{rex}\times \AbCat_{\bk}^{rex}\;\to\; \AbCat_{\bk}^{rex}$$
    on $\Tak^{ex}_{\bk}\times \AbCat^{rex}_{\bk}$.
\end{proposition}
\begin{proof}
By the above, $\Tak_{\bk}^{ex}$ is equivalent to the 2-category $\CoAl$ which has as objects coalgebras over $\bk$, as 1-morphisms bicomodules that are injective as right comodules and as 2-morphisms the morphisms of bimodules. We can replace $\Tak_{\bk}^{ex}$ by $\CoAl$ for the purposes of the proposition.

    For $C\in\CoAl$ and $\cB\in \AbCat$ we then set (following \Cref{LemExDTP}(1))
    $$C\boxtimes\cB:= \Comod_C\cB.$$

    For all $C_1,C_2\in\CoAl$ and $\cB_1,\cB_2\in \AbCat$ we must now define a functor
    $$\CoAl(C_1,C_2)\times \AbCat(\cB_1,\cB_2)\;\to\; \AbCat(\Comod_{C_1}\cB_1,\Comod_{C_2}\cB_2).$$
    For $M\in \CoAl(\cA_1,\cA_2)$ and $G\in \AbCat(\cB_1,\cB_2)$, we define the composite functor
    $$M\boxtimes G:\; \Comod_{C_1}\cB_1\xrightarrow{G} \Comod_{C_1}\cB_2\xrightarrow{M\Box^{C_1}-} \Comod_{C_2}\cB_2,$$
    where we use the notation $G$ for the canonical lift of $G$ to the categories of comodules.
    Finally, for a morphism $f:M\to M'$ and a natural transformation $\beta: G\Rightarrow G'$, the corresponding natural transformation from $M\boxtimes G$ to $M'\boxtimes G'$ is easily defined.

    That this definition satisfies all properties of a pseudo-functor and agrees with the classical interpretation of $-\boxtimes-$ as a pseudo-functor is left as an exercise.
\end{proof}

\begin{remark}
     One can verify that we cannot further extend $-\boxtimes-$ to a pseudo-functor
     $$-\boxtimes-:\;\Tak_{\bk}^{rex}\times \AbCat_{\bk}\;\to\; \AbCat_{\bk}.$$
     Indeed, the functors $H_0(G_1,-)$ and $H^0(G_2,-)$ do not commute on $G_1\times G_2$-representations, for finite groups $G_1,G_2$.
\end{remark}

\begin{corollary}
Take $\cA,\cB,\cB'$ in $\Tak$.
    For a $\bk$-linear functor $F:\cB\to\cB'$, the functor $\cA\boxtimes F$ is an equivalence if and only if $F$ is an equivalence.
\end{corollary}
\begin{proof}
    If $F$ is an equivalence, it follows that so is $\cA\boxtimes F$. Conversely, assume that $\cA\boxtimes F$ is an equivalence, and choose a coalgebra $C$ for $\cA$ as in \ref{DefDel2}.

For a given object $X$ in $\cA$ (interpreted as a comodule), we have a commutative square
$$\xymatrix{
\Comod_C\cB\ar[rr]^{\cA\boxtimes F}&&\Comod_C\cB'\\
\cB\ar[u]^{X\boxtimes -}\ar[rr]^{F}&& \cB'\ar[u]^{X\boxtimes -}.
    }$$
    For $A,B\in\cB$, the equivalence $\cA\boxtimes F$ induces isomorphisms
    $$\End_{C}(X)\otimes\Hom_{\cB}(A,B)\xrightarrow{\sim}\End_C(X)\otimes\End_{\cB'}(FA,FB),$$
    showing that $F$ is fully faithful.

   Now we consider the commutative square
    $$\xymatrix{
\Comod_C\cB\ar[rr]^{\cA\boxtimes F}\ar[d]&&\Comod_C\cB'\ar[d]\\
\cB\ar[rr]^{F}&& \cB',
    }$$
    where the downwards arrows are forgetful functors. For any object $Y\in\cB'$, there is some $l\in\mZ_{>0}$ such that $Y^{\oplus l}$ is in the image of the downwards arrow, and hence in the image of $F$. Since $F$ is fully faithful and $\cB'$ is Krull-Schmidt, it follows that also $F$ is essentially surjective.
\end{proof}

The following is a standard application of \Cref{LemComod}(2), which shows that one can refine the defining property of the Deligne product, from right exact functors to exact functors.
\begin{lemma}\label{LemDelEx}
    For $\cA,\cB,\cC$ in $\Tak$, restriction along $\cA\times\cB\to\cC$ yields an equivalence between the category of exact $\bk$-linear functors $\cA\boxtimes \cB\to\cC$ with the category of bilinear bifunctors $\cA\times\cB\to\cC$ that are exact in each variable. 
\end{lemma}

\subsection{The Deligne product of tensor categories}

In analogy with the notation $\Tak^{ex}_{\bk}$, we denote by $\Tens_{\bk}^{ex}$ the 2-category of tensor categories, tensor functors and natural transformations of monoidal functors and by $\Tens_{\bk}$ the 2-category of tensor categories, LSM functors and natural transformations of monoidal functors. We have obvious forgetful 2-functors to $\Tak^{ex}_{\bk}$ and $\Tak_{\bk}$, respectively. Note also that LSM functors between tensor categories are exact whenever they are right exact, by duality.

It is a classical fact from \cite[\S5]{Del90} that $-\boxtimes-$ from $\Tak^{rex}\times\Tak^{rex}$ to $\Tak^{rex}$ lifts to a pseudo-functor
\begin{equation}\label{boxTens1}
    -\boxtimes-:\;\Tens_{\bk}^{ex}\times\Tens_{\bk}^{ex}\;\to\; \Tens_{\bk}^{ex}.
\end{equation}
We sketch a proof that allows us to generalise again to include non-exact functors.

\subsubsection{}
Consider tensor categories $\cC_1,\cC_2$ over $\bk$. The tensor structure on $\cC_i$ equips the coalgebra $C_i$ (obtained from some choice of exact faithful $\cC_i\to\Vecc$) with the structure of a copseudo-Hopf algebra, see \cite[\S 2.5]{CEO}. We can use this to equip $\cC_1\boxtimes\cC_2$ with a unique structure of a tensor category for which the functors
\begin{equation}\label{eqC1C2}-\boxtimes \unit:\cC_1\to\cC_1\boxtimes\cC_2\quad\mbox{and}\quad\unit\boxtimes -:\cC_2\to\cC_1\boxtimes\cC_2\end{equation}
are tensor functors, \cite[Proposition~5.17]{Del90}.
Indeed, it follows quickly, see \cite[\S 5.16]{Del90}, that there is a unique symmetric tensor product on $\cC_1\boxtimes\cC_2$ compatible with $\cC_1\times\cC_2\to\cC_1\boxtimes\cC_2$, the only subtlety is proving that the tensor product is exact. This is equivalent to $\cC_1\boxtimes\cC_2$ being rigid, via \Cref{LemExDTP}(2), and only known to be true when $\bk$ is perfect, see the discussion in~\cite[\S 6]{CEOP}. The copseudo-Hopf algebra structure on $C=C_1$ contains a $(C,C^{\otimes 2})$-bicomodule~$T$, see \cite[Definition~2.1]{CEO}, so that the tensor product on $\Comod_C(\cC_2)$ is given by
\begin{equation}\label{TensProdBox}
M\times N\;\mapsto\; T\Box^{C\otimes C}(M\otimes N),\end{equation}
where the occurrence of $\otimes$ refers to the tensor product in $\cC_2$ of (the objects in $\cC_2$ underlying) $M$ and $N$. The crucial observation 
is therefore \Cref{LemComod}(2) from which we can indeed derive exactness of the tensor product.

\subsubsection{}\label{somesubsec} We can further show that the pseudo-functor in \Cref{boxTens1} extends to a pseudo-functor
\begin{equation}\label{boxTens2}
    -\boxtimes-:\;\Tens_{\bk}^{ex}\times\Tens_{\bk}\;\to\; \Tens_{\bk}
\end{equation}
which is compatible with the pseudo-functor in \Cref{PropPseudo}. Indeed, since $T$ is injective as a right ${C\otimes C}$-comodule, the operation
$$M\otimes N\;\leadsto\; T\Box^{C\otimes C}(M\otimes N)$$
is an additive one, which commutes with $\bk$-linear functors, not only exact ones. 

An instance of the coherence conditions of a pseudo-functor is the following example.

\begin{example}
\label{Lem:Comm}
    Consider a tensor functor $F:\cC_1\to\cC_2$ and an LSM functor $\Theta:\cD_1\to\cD_2$, the following diagram of LSM functors is commutative
    $$\xymatrix{
\cC_1\boxtimes\cD_1\ar[rr]^{\cC_1\boxtimes\Theta}\ar[d]^{F\boxtimes \cD_1}&&\cC_1\boxtimes\cD_2\ar[d]^{F\boxtimes \cD_2}\\
    \cC_2\boxtimes\cD_1\ar[rr]^{\cC_2\boxtimes\Theta}&&\cC_2\boxtimes\cD_2.
    }$$
\end{example}

\begin{lemma}\label{Cor:box}
    Consider tensor categories $\cC_1,\cC_2,\cC$ over $\bk$. If 
    $$F:\cC_1\boxtimes\cC_2\;\to\;\cC$$
    is an LSM functor that restricts under the functors in \Cref{eqC1C2} to tensor functors on $\cC_1,\cC_2$, then $F$ is exact, {\it i.e.}~a tensor functor.
\end{lemma}
\begin{proof}
    By \Cref{Thm:FE} it is sufficient to show that $F$ is faithful. Since every object in $\cC_1\boxtimes\cC_2$ is a subobject and a quotient of an object $X\boxtimes Y$, it suffices to prove that no non-zero morphism
    $$X\boxtimes Y\;\to\;U\boxtimes V$$
    is sent to zero by $F$. However, the restriction of $F$ to such objects corresponds to that of the (faithful) tensor functor $\cC_1\boxtimes\cC_2\to\cC$ given by the universal property of $\cC_1\boxtimes\cC_2$ from the two restrictions of $F$.
\end{proof}

\subsubsection{}
For future use, we also record the following result.
Consider a tensor category $\cC$ and $\cA,\cB\in\Tak$. Consider the functor
$$\cC\times\cC\times \cA\to\cC\boxtimes\cA,\quad (X_1,X_2,A)\mapsto (X_1\otimes X_2)\boxtimes \cA$$
which is $\bk$-linear and exact in each variable. We can reinterpret this as a functor
$$\cC\times\cA\;\to\;\Fun_{\bk}(\cC,\cC\boxtimes \cA),$$
leading via \Cref{LemDelEx} to an exact functor with domain $\cC\boxtimes\cA$. In turn we can interpret this as a functor 
$$-\circledast-:\cC\times (\cC\boxtimes \cA)\;\to\;\cC\boxtimes\cA$$
which is linear and exact in both variables. Such functors make $\cC\boxtimes\cA$ and $\cC\boxtimes \cB$ into module categories over $\cC$ in the sense of \cite[\S 7.3]{EGNO}.

\begin{proposition}
    For a $\bk$-linear functor $\Theta:\cA\to\cB$, the $\bk$-linear functor $\cC\boxtimes\Theta$ is a $\cC$-module functor. 
\end{proposition}
\begin{proof}
    This follows from the commutative diagram
    $$\xymatrix{
\cC\boxtimes\cC\boxtimes\cA\ar[rr]^{(\cC\boxtimes\cC)\boxtimes\Theta}\ar[d]^{\Delta\boxtimes \cA}&& \cC\boxtimes\cC\boxtimes\cB\ar[d]^{\Delta\boxtimes \cB}\\
    \cC\boxtimes\cA\ar[rr]^{(\cC\boxtimes\cC)\boxtimes\Theta}&& \cC\boxtimes\cB,
    }$$
    where $\Delta:\cC\boxtimes\cC\to\cC$ is the exact functor corresponding to $-\otimes-:\cC\times\cC\to\cC$ under \Cref{LemDelEx}.
\end{proof}

\subsection{The tensor power functor}\label{SecTPF}
Let $\cC$ be a tensor category and fix $n\in\mZ_{>0}$.

\subsubsection{}
The category $\cC\boxtimes\Rep S_n$, see \Cref{ExRepG}(1), is the tensor category of $S_n$-representations in $\cC$. The canonical example of an $S_n$-representation in $\cC$ is $X^{\otimes n}$ equipped with the braiding action. It follows that
\begin{equation}\label{eqTPF}
    \cC\to\cC\boxtimes \Rep S_n,\quad X\mapsto X^{\otimes n}
\end{equation}
is a well-defined symmetric monoidal (non-additive) functor.

\begin{lemma}\label{LemPoFu}
    For every tensor functor $F:\cC\to\cD$, the diagram of symmetric monoidal functors
    $$\xymatrix{
    \cC\ar[rr]^-{(-)^{\otimes n}}\ar[d]^F&&\cC\boxtimes\Rep S_n\ar[d]^{F\boxtimes\Rep S_n}\\
    \cD\ar[rr]^-{(-)^{\otimes n}}&&\cD\boxtimes\Rep S_n
    }$$
    is commutative.
\end{lemma}

\begin{remark}\label{RemHomPP}
Assume that $\cC$ has projective objects. Then we have the exact faithful functor
$$\omega=\bigoplus_P\Hom_{\cC}(P,-):\;\cC\to\Vecc,$$
where $P$ ranges over all projective covers of simple objects in $\cC$.
We can then construct the following functor
$$\cC\xrightarrow{(-)^{\otimes n}}\cC\boxtimes\Rep S_n\xrightarrow{\omega\boxtimes\Rep S_n}\Rep S_n,\quad X\;\mapsto\;\bigoplus_P\Hom_{\cC}(P, X^{\otimes n}),$$
where we consider the obvious $S_n$-module structures.
\end{remark}


\section{\texorpdfstring{$\bO$}{O}-functors}\label{SecO}
Let $\bk$ be an algebraically closed field of characteristic $p>0$ and fix $i\in\mZ$.

\subsection{Definitions} \label{sec:definitions}

Let $\cV$ be a tensor category (over $\bk$).

\begin{definition} \label{DefO}
\begin{enumerate}
    \item An $\bO^i\{\cV\}$-functor\footnote{We use the letter $\bO$ as a reference to Victor Ostrik, who in \cite{Os} was the first to introduce, and demonstrate the central role of, the first example of an $\bO^i\{\cV\}$-functor, the Frobenius functor.} $\Phi$ is a family of $\Frob^i$-LSM functors
    $$\Phi_{\cC}\;:\; \cC\to\cC\boxtimes\cV,$$
    for every tensor category $\cC$ over $\bk$, such that, for each tensor functor $F\colon\cC\to\cD$, the diagram of LSM functors
    $$\xymatrix{
    \cC\ar[rr]^-{\Phi_{\cC}}\ar[d]^F&&\cC\boxtimes\cV\ar[d]^{F\boxtimes\cV}\\
    \cD\ar[rr]^-{\Phi_{\cD}}&&\cD\boxtimes\cV
    }$$
    is commutative up to a natural isomorphism (of monoidal functors).
    \item Given an $\bO^i\{\cV\}$-functor $\Phi$, a tensor category $\cC$ is called {\bf $\Phi$-exact} if $\Phi_{\cC}$ is exact, and hence a $\Frob^i$-tensor functor. 
    \item Given an $\bO^i\{\cV\}$-functor $\Phi$, the {\bf $\Phi$-type} of $\cC$ is the minimal tensor subcategory $\cV_1\subset\cV$ for which $\Phi_{\cC}$ takes values in
$$\cC\boxtimes\cV_1\;\subset\; \cC\boxtimes\cV.$$
\end{enumerate}
\end{definition}

\begin{remark}
    One could make the definition of an $\bO^i\{\cV\}$-functor stricter by adding the natural transformations for the commutative squares to the data and imposing coherence conditions. Concretely, we can consider the forgetful 2-functor from the 2-category of tensor categories and tensor functors to the 2-category of tensor categories and additive symmetric monoidal functors, and the composite of this 2-functor with the pseudo-functor $-\boxtimes\cV$. The stricter definition of an $\bO^i\{\cV\}$-functor would be a pseudo-natural transformation (consisting of $\Frob^i$-linear functors) between these pseudo-functors.
\end{remark}

\begin{lemma}\label{Ofunfun}
    Consider a tensor functor $F:\cC'\to\cC$ and an $\bO^i\{\cV\}$-functor $\Phi$.
    \begin{enumerate}
        \item If $\cC$ is $\Phi$-exact, then also $\cC'$ is $\Phi$-exact.
        \item Let $\cV_1$ and $\cV_1'$ be the $\Phi$-types of $\cC$ and $\cC'$, then $\cV_1'\subset\cV_1$.
    \end{enumerate}
\end{lemma}
\begin{proof}
    Part (1) follows from commutativity of the diagram in \Cref{DefO}(1), and using either \Cref{Thm:FE} or the fact that $F\boxtimes\cV$ is exact in \Cref{DefO}(1), by \Cref{boxTens1}.

    Also part (2) follows from commutativity of the diagram in \Cref{DefO}(1), for instance by viewing $\cC'$ as a category of representations over some affine group scheme in $\cC$ as in \cite[\S 7]{Del90}.
\end{proof}

\subsubsection{}\label{ProcDesc}
Given a second tensor category $\cV'$ and an LSM-functor $F:\cV\to\cV'$, using \Cref{Lem:Comm}, we obtain an $\bO^i\{\cV'\}$-functor from any $\bO^i\{\cV\}$-functor $\Phi$, by considering the composites
$$\Phi_{\cC}':\;\cC\xrightarrow{\Phi_{\cC}}\cC\boxtimes\cV\xrightarrow{\cC\boxtimes F}\cC\boxtimes\cV'.$$
If $F$ is exact ({\it i.e.}~a tensor functor), we call the resulting $\bO^i\{\cV'\}$-functor $\Phi'$ a {\bf descendant} of $\Phi$.

If $\Phi'$ is a descendant of $\Phi$, then $\cC$ is $\Phi$-exact if and only if it is $\Phi'$-exact. Partly because of this, it is natural to focus on $\bO^i\{\cV\}$-functors for incompressible~$\cV$. In particular, we \emph{abbreviate $\bO^i\{\Ver_{p^n}\}$ to $\bO^i_n$.}

\begin{example} \label{ExOFirst}
    \begin{enumerate}
        \item The standard inclusion $\cC\simeq\cC\boxtimes\Vecc\hookrightarrow\cC\boxtimes\Ver_{p^n}$ is a trivial example of an $\bO^0_n$-functor.
         
        \item The Frobenius functor
        $$\Fr:\cC\to\cC\boxtimes\Ver_p$$
        from \cite{Tann, EOf, Os} is an $\bO^1_1$-functor. Usually $\Fr$-exact categories are called Frobenius exact. The Frobenius functor $\Fr$ is also a descendant of the enhanced Frobenius functor $\Fr^{en}$, see \cite[\S 3.2]{CEO}, which is an $\bO^1\{\overline{\Rep S_p}\}$-functor.
       
        \item Assume that we are given an $\bO^i\{\cV\}$-functor $\Phi$. Consider the composite LSM functor
        $$H:\cV\xrightarrow{\Sigma}\overline{\cV}\to \Ver_p,$$
        with the second functor the unique tensor functor to $\Ver_p$, guaranteed to exist by the $\Ver_p$-Theorem (\Cref{pTheorem}). Then the procedure in \Cref{ProcDesc} yields the~$\bO^i_1$-functor~$\overline{\Phi}$
        \begin{equation}\label{eq:Phibar}
            \overline{\Phi}_{\cC}: \; \cC\xrightarrow{\Phi_{\cC}}\cC\boxtimes\cV\xrightarrow{\cC\boxtimes H}\cC\boxtimes\Ver_p.
        \end{equation}
        Note that $\overline{\Phi}$ is a descendant of $\Phi$ if and only if $\cV$ is semisimple.
    \end{enumerate}
\end{example}

\begin{lemma}\label{Lem:DefO} Let $\Phi$ be an $\bO^i\{\cV\}$-functor, and $\cC$ and $\cD$ tensor categories.
    \begin{enumerate}
        \item If $\cC$ and $\cD$ are $\Phi$-exact, then so is $\cC\boxtimes\cD$.
        \item If $\cC$ is $\overline{\Phi}$-exact, then it is also $\Phi$-exact.
    \end{enumerate}
\end{lemma}
\begin{proof}
    Part (1) is an example of \Cref{Cor:box}. For part (2), it suffices (using \Cref{Thm:FE}) to observe that faithfulness of the composite \Cref{eq:Phibar} implies in particular that $\Phi_{\cC}$ is faithful.
\end{proof}

\subsection{Potential properties}

Our main motivation to introduce and study $\bO$-functors is their potential for replacing the Frobenius functor in the proof of the $\Ver_p$-Theorem towards proving the $\Ver_{p^\infty}$-Conjecture.
In this section, we therefore explore potential properties of $\bO$-functors that would be required to realise this. We fix $n\in\mZ_{>0}$ and an $\bO^i_n$-functor $\Phi$.

\subsubsection{}We have the following potential properties of an $\bO^i_n$-functor $\Phi$:
\begin{enumerate}
    \item[(O0)] $\Ver_{p^n}$ is $\Phi$-exact.
    \item[(O1)] We have $\Phi_{\Ver_{p^n}}(L_1)\simeq\unit\boxtimes L_1$ or (if $p>2$) $\Phi_{\Ver_{p^n}}(L_1)\simeq\bar{\unit}\boxtimes (\bar{\unit}\otimes L_1)$.
    \item[(O2)] A tensor category of moderate growth is $\overline{\Phi}$-exact if and only if it is $\Fr$-exact.
    \item[(O3)] Let $\cC_1\to\cC_2$ be a surjective tensor functor, and assume that $\cC_1$ is $\Phi$-exact and of moderate growth. Then also $\cC_2$ is $\Phi$-exact.
\end{enumerate}

As will be explained in the following subsection, the above technical properties are motivated by the below conceptual potential property, that $\Phi$ determines which tensor categories `fibre' over $\Ver_{p^n}$:
\begin{enumerate}
    \item[(OF)] The following properties are equivalent on a tensor category $\cC$:
    \begin{enumerate}
        \item[(1)] $\cC$ is of moderate growth and $\Phi$-exact;
        \item[(2)] $\cC$ admits a tensor functor to $\Ver_{p^n}$.
    \end{enumerate}
    \end{enumerate}

\begin{remark}\label{RemDefO}
\begin{enumerate}
\item The $\Ver_p$-Theorem states precisely that $\Fr$ satisfies (OF) for $n=1$.
\item By \Cref{CorVVV}, (O1) implies (O0). Moreover, (O1) does not make sense for $p^n=2$, so by convention for that case we redefine (O1) as (O0).

    \item LSM functors are exact on semisimple tensor categories. 
    It thus follows from the $\Ver_p$-Theorem that a Frobenius exact tensor category of moderate growth is $\Psi$-exact, for every $\bO^i\{\cV\}$-functor $\Psi$. In particular, the `if and only if' condition in (O2) is always satisfied in one direction.
    \item Condition (O3) is automatically satisfied when we only consider tensor categories with property (SP). Indeed, this follows from \Cref{ExFin}.
    That Condition (O3) is satisfied in general for the Frobenius functor was proved in \cite{CEO2}.
\end{enumerate}
    
\end{remark}

\subsection{Interpretation of the potential properties} \label{SecOProp}

In this section, we fix $n\in\mZ_{>0}$ and an $\bO^i_n$-functor $\Phi$.

\subsubsection{Summary}
In the quest for $\bO$-functors satisfying (OF), imposing Condition (O1) is natural. Indeed, by \Cref{CorVVV} it is one of only two options. It is the option that is satisfied by the Frobenius functor, and one of the indispensable properties exploited in \cite{CEO} for proving the $\Ver_p$-Theorem. Once we accept (O1) and (OF), condition (O2) follows, by \Cref{Thm:ReverseO} below.

On the other hand, assumptions (O1) and (O2) are also sufficient to imply (OF).
This is proved in \Cref{ThmO} below, which contains the $\Ver_p$-Theorem as a special case, but is proved by reducing to that case.

The results in this section thus motivate the quest for $\bO^i_n$-functors satisfying (O1) and (O2). Moreover, \Cref{PropO1} below motivates the quest for $\bO^i_n$-functors satisfying simply (O1); and \Cref{EsotericProp} in the next section even motivates the quest for arbitrary $\bO^i_n$-functors.

\begin{theorem}\label{Thm:ReverseO}
    Assume that $\Phi$ satisfies {\rm (OF)}, then:
    \begin{enumerate}
    \item $\Phi$ satisfies {\rm (O0)}.
    \item If $\Phi$ satisfies {\rm (O1)}, then it also satisfies {\rm (O2)}.
        \item Condition {\rm (O3)} on $\Phi$ is equivalent to the condition that $\Ver_{p^n}$ be Bezrukavnikov, as defined in \cite{CEO2}.
        
    \end{enumerate}
\end{theorem}
\begin{proof}
    Parts (1) and (3) follow by definition. Part (2) for $n=1$ is a tautology by the $\Ver_p$-theorem, so we focus on $n>1$. Consider a tensor category~$\cC$ that is $\overline{\Phi}$-exact. 
    \Cref{Lem:DefO}(2) implies that $\cC$ is also $\Phi$-exact and so by assumption admits a tensor functor $F$ to $\Ver_{p^n}$.
    We have the commutative diagram
    $$\xymatrix{
\cC\ar[rr]^{\overline{\Phi}_{\cC}}\ar[d]^{F} &&\cC\boxtimes \Ver_p\ar[d]^{F\boxtimes \Ver_{p^n}}\\
\Ver_{p^n}\ar[rr]^{\overline{\Phi}_{\Ver_{p^n}}}&& \Ver_{p^n}\boxtimes\Ver_p,
    }$$
    and by assumption (O1) on $\Phi$ (and \Cref{CorVVV}(1)), $\overline{\Phi}_{\Ver_{p^n}}$ takes values in the semisimple tensor category $\sVec\boxtimes\Ver_p$. It follows that the upper path represents a $\Frob^i$-tensor functor from~$\cC$ to a semisimple tensor category, so~$\cC$ is indeed Frobenius exact.
\end{proof}

\begin{theorem}\label{ThmO}
    If $\Phi$ satisfies {\rm (O1)} and {\rm (O2)}, then the following are equivalent for a tensor category $\cC$ over $\bk$:
    \begin{enumerate}
        \item $\cC$ is $\Phi$-exact and of moderate growth;
        \item $\cC$ admits a tensor functor to $\Ver_{p^n}$.
    \end{enumerate}
    In other words, properties {\rm (O1)} and {\rm (O2)} together imply {\rm (OF)}.
\end{theorem}
\begin{proof}
    That (2) implies (1) follows from assumption (O1). For convenience, we write out the proof that (1) implies (2) for $p>2$ and for the option
    $$\Phi(L_1)\;=\; \bar{\unit}\boxtimes (\bar{\unit}\otimes L_1)$$ in (O1), the remaining case being identical up to notation.

    Applying the commutative diagram in \Cref{DefO} to $F=\Phi_{\cC}$ shows that the two composites in
    \begin{equation}\label{Eqboxbox}
        \xymatrix{
    \cC\ar[rr]^-{\Phi_{\cC}}&& \cC\boxtimes\Ver_{p^n} \ar@<-.5ex>[rr]_-{\Phi_{\cC}\boxtimes\Ver_{p^n}} \ar@<.5ex>[rr]^-{\Phi_{\cC\boxtimes \Ver_{p^n}}} &&\cC\boxtimes\Ver_{p^n}\boxtimes\Ver_{p^n}.
    }
    \end{equation}
    are isomorphic. By assumption (O1), the upper right functor takes values in
    $$\cC\boxtimes\sVec\boxtimes \Ver_{p^n}.$$
    Indeed, this is true for objects $X\boxtimes Y$ in $\cC\boxtimes \Ver_{p^n}$ by \Cref{CorVVV}(1), and hence true for all objects by exactness of $\Phi_{\cC\boxtimes\Ver_{p^n}}$ in \Cref{Lem:DefO}(1).
Therefore both composites in \Cref{Eqboxbox} also take values in this tensor subcategory.

    The simple objects in $\cC\boxtimes\Ver_{p^n}$ are of the form $V_1\boxtimes V_2$, for simple objects $V_1\in\cC$, $V_2\in \Ver_{p^n}$.
    Let $V\in\cC$ be any simple object for which there exists a simple $V'\in\Ver_{p^n}$, for which $V\boxtimes V'$ appears as a simple constituent (in the Jordan-H\"older series) of $\Phi_{\cC}(X)$ for some~$X\in\cC$. Since $\Phi_{\cC}\boxtimes\Ver_{p^n}$ is exact, it thus follows from the previous paragraph that 
    $$\Phi_{\cC}(V)\;\in\; \cC\boxtimes \sVec\;\subset\; \cC\boxtimes\Ver_{p^n}.$$

    Let $\cC_1$ denote the Serre tensor subcategory of $\cC$ generated by all simple $V\in\cC$ as above. It then follows that $\Phi_{\cC}$ takes values in $\cC_1\boxtimes \Ver_{p^n}$ and that $\Phi_{\cC_1}$ takes values in $\cC_1\boxtimes\sVec$ (that is, the $\Phi$-type of $\cC_1$ is $\Vecc$ or $\sVec$).

     Next we can observe that $\cC_1$ is $\overline{\Phi}$-exact. Indeed
    $$\sVec\hookrightarrow \Ver_{p^n}\to\overline{\Ver_{p^n}}\to\Ver_p$$
    compose to the (exact) inclusion $\sVec\subset\Ver_p$. Consequently the composite $\overline{\Phi}_{\cC_1}$
    $$\xymatrix{\cC_1\ar[rr]^{\Phi_{\cC_1}}\ar@{-->}[rrd]&& \cC_1\boxtimes\Ver_{p^n}\ar[rr]^{\cC_1\boxtimes H}&&\cC_1\boxtimes \Ver_p\\
    && \cC_1\boxtimes\sVec\ar@{^{(}->}[u]
    }$$
    is indeed exact. By Assumption (O2) and the $\Ver_p$-Theorem, we therefore have a tensor functor $F:\cC_1\to\Ver_p$, yielding a $\Frob^i$-tensor functor
    $$\cC\xrightarrow{\Phi_{\cC}}\cC_1\boxtimes\Ver_{p^n}\xrightarrow{F\boxtimes \Ver_{p^n}}\Ver_p\boxtimes\Ver_{p^n}\xrightarrow{\otimes}\Ver_{p^n}.$$
    This concludes the proof, since by \Cref{RemFrobVer}, we can postcompose with a $\Frob^{-i}$-tensor functor $\Ver_{p^n}\to\Ver_{p^n}$.
\end{proof}

The following proposition is an example of the potential of $\bO$-functors in the structure theory of tensor categories. For the moment, we can only prove such a result for finite tensor categories.
\begin{proposition}
    \label{PropO1}
    Assume that the $\Ver_{p^\infty}$-Conjecture (see \Cref{BEOConjecture}) is true, and that $p^n>2$. If $\Phi$ satisfies {\rm (O1)}, then the following are equivalent on a {\rm finite} tensor category $\cC$:
    \begin{enumerate}
        \item $\cC$ is $\Phi$-exact;
        \item $\cC$ admits a tensor functor to $\Ver_{p^n}$.
    \end{enumerate}
\end{proposition}

As part of the proof of \Cref{PropO1}, we derive the following unconditional lemma.

\begin{lemma}\label{Lem+}
    Assume that $p^n>2$. Let $\Phi$ be an $\bO^i_n$-functor satisfying {\rm (O1)}. Then $\Ver_{p^{n+1}}^+$ is not $\Phi$-exact.
\end{lemma}
\begin{proof}
    Assume first that $p>2$. Then $\Ver_{p^{n+1}}^+$ is generated by the simple object 
    $$L:=L_1^{[n+1]}\otimes \bar{\unit}= L_{1+p^{n}(p-2)}^{[n+1]}.$$ 
    Assume for a contradiction that $\Phi_{\Ver_{p^{n+1}}^+}$ is exact. Since there is only one object with corresponding Frobenius--Perron dimension in the target, see \cite[Corollary~4.44]{BEO}, we find that
    $$\Ver_{p^{n+1}}^+\xrightarrow{\Phi}\Ver_{p^{n+1}}^+\boxtimes \Ver_{p^n}, \quad L\mapsto L\boxtimes \unit.$$
    This implies that the functor is an equivalence onto the subcategory $\Ver_{p^{n+1}}^+\boxtimes \Vecc$. However, this contradicts assumption (O1) applied to $\Ver_{p^n}\subset \Ver_{p^{n+1}}^+$.


    Now consider $p=2$ and hence $n>1$. It follows from Assumption (O1) that, under
    $$\Ver_{2^{n+1}}^+\;\xrightarrow{\Phi_{\Ver_{2^{n+1}}^+}}\;\Ver_{2^{n+1}}^+\boxtimes \Ver_{2^n},$$
    the generator $L_1$ of $\Ver_{2^n}\subset\Ver_{2^{n+1}}^+$ is sent to $\unit\boxtimes L_1$.

    On the other hand, $\Ver_{2^{n+1}}^+$ is incompressible, so the displayed functor is an equivalence onto a tensor subcategory of the right-hand side. We also observe that $\Ver_{2^{n+1}}^+$ is indecomposable as an additive category, see \cite[Theorem~2.1(viii)]{BE}, whereas the right-hand side decomposes into the tensor subcategory $\Ver_{2^{n+1}}^+\boxtimes \Ver_{2^n}^+$ and a subcategory equivalent to $\Ver_{2^{n+1}}^+\boxtimes \Ver_{2^{n-1}}$, see \cite[Theorem~2.1(ix)]{BE}. Hence, we find that the displayed functor takes values in $\Ver_{2^{n+1}}^+\boxtimes \Ver_{2^n}^+$, which contradicts the conclusion from the previous paragraph.
\end{proof}

\begin{lemma}\label{LemOCond}
    Assume that the $\Ver_{p^\infty}$-Conjecture is true and that ${\Ver_{p^n}}$ is $\Phi$-exact ({\it i.e.}~$\Phi$ satisfies (O0)), but $\Ver_{p^{n+1}}^+$ is not $\Phi$-exact. Then a finite tensor category is $\Phi$-exact if and only if it admits a tensor functor to $\Ver_{p^n}$.
\end{lemma}
\begin{proof}
    This follows from \Cref{CorExFin} and the classification of tensor subcategories of $\Ver_{p^n}$ in \cite[Corollary~4.61]{BEO}.
\end{proof}

    \begin{proof}[Proof of \Cref{PropO1}]
        The conclusion follows immediately from the combination of \Cref{Lem+} and \Cref{LemOCond}.
    \end{proof}

    \subsection{\texorpdfstring{$\bO$}{O}-functors and incompressible categories}
Here we show that to any $\bO$-functor $\Phi$ (without further assumptions), we can associate an incompressible category so that $\Phi$ controls which tensor categories fibre over it.

    \subsubsection{} By a {\bf class of incompressible categories}, we mean a family $\{\cC_\alpha\}$ of incompressible categories for which $\cC_\alpha$, for every $\alpha$, is equivalent to a tensor subcategory of $\cC_\beta$ for every~$\beta$. 
    
    The $\Ver_{p^\infty}$-Conjecture predicts in particular that classes of incompressible categories of moderate growth are simply {\em equivalence} classes.

    \begin{proposition}\label{EsotericProp}
        Let $\cV$ be a tensor category and $\Phi$ an $\bO^i\{\cV\}$-functor. There exists a unique class of incompressible categories $\{\cC_\alpha\}$, all satisfying {\rm (SP)}, for which the following properties are equivalent on every {\bf finite} tensor category $\cC$:
        \begin{enumerate}
            \item $\cC$ is $\Phi$-exact;
            \item $\cC$ admits a tensor functor to some $\cC_\alpha$;
            \item $\cC$ admits a tensor functor to every $\cC_\alpha$.
        \end{enumerate}
    \end{proposition}
\begin{proof}
    Consider the collection of all finite incompressible tensor categories which are $\Phi$-exact. It is non-empty as it contains $\Vecc$. By \cite[Corollary~5.2.8]{CEO2}, there exists an incompressible tensor category $\cD$ which contains all these finite incompressible categories as tensor subcategories. If necessary, we replace $\cD$ by its minimal tensor subcategory which contains all these finite subcategories. 
    By construction $\cD$ satisfies (SP) and is $\Phi$-exact. 

Now let $\cC$ be an arbitrary $\Phi$-exact finite tensor category. By \cite[Theorem~5.2.1]{CEO2}, $\cC$ admits a surjective tensor functor to an incompressible tensor category. By surjectivity, this tensor category is finite. By \Cref{RemDefO}(3), this finite tensor category is also $\Phi$-exact. By the previous paragraph, $\cC$ thus admits a tensor functor to $\cD$. In conclusion a finite tensor category is $\Phi$-exact if and only if it admits a tensor functor to $\cD$.

Now we let $\{\cC_\alpha\}$ be the unique class of incompressible categories that $\cD$ belongs to. This is the collection of those incompressible tensor subcategories of $\cD$ which themselves also contain a tensor subcategory equivalent to $\cD$. The statement of the proposition now follows.
\end{proof}

\begin{remark}
    If we assume that $\Phi$ satisfies (O3), or strictly speaking its immediate generalisation from $\bO^i_n$-functors to $\bO^i\{\cV\}$-functors, we can write a version of \Cref{EsotericProp} without finiteness assumptions, with identical proof.
\end{remark}

\section{Construction of \texorpdfstring{$\bO$}{O}-functors}\label{SecConst}

Let $p$ be a prime and $\bk$ an algebraically closed field of characteristic $p$.

\subsection{\texorpdfstring{$\bV$}{V}-functors}\label{DefV}

Let $G<\Perm(W)$ be a permutation group, for a finite set $W$. 

\begin{definition} \label{def::V-functor}
    A $\bV$-functor for $G<\Perm(W)$ is an LSM functor 
$$\Theta:\Rep_{\bk}G\;\to\; \cV$$
to a tensor category $\cV$ over $\bk$ with the condition that for every non-transitive subgroup $H<G$,
$$\Theta(\Ind^G_{H}\unit)=0.$$
\end{definition} 

Given any tensor functor $F:\cV\to\cV'$ and $\bV$-functor $\Theta:\Rep G\to\cV$, we obtain a $\bV$-functor $F\circ\Theta$, which we call a {\bf descendant} of $\Theta$.

\begin{remark}\label{RemHProj}
    A $G$-module is called \emph{$H$-projective} if it appears as a direct summand of a $G$-module which is induced from $H$. It follows that an LSM functor to a tensor category is a $\bV$-functor if and only if it annihilates all $H$-projective modules for all non-transitive $H<G$.
\end{remark}

\begin{lemma}\label{SMGExists}
The following conditions are equivalent on $G<\Perm(W)$.
\begin{enumerate}
    \item There exist $\bV$-functors for $G$;
    \item The semisimplification $\Sigma:\Rep G\to\overline{\Rep G}$ is a $\bV$-functor;
    \item There is a $p$-subgroup $P<G$ of $G$ that acts transitively on $W$;
    \item $|W|$ is a power of $p$ and $G$ acts transitively on $W$.
\end{enumerate}

\end{lemma}

\begin{proof}
First we prove that (1) implies (3). Let $P<G$ be a Sylow $p$-subgroup. Then $\Ind^G_P\unit$ has categorical dimension different from zero, and cannot be sent to zero under an LSM functor. Under assumption (1) it thus follows that $P$ acts transitively on $W$, implying (3).

Obviously, (2) implies (1). 

    We prove that (3) implies (2). It follows from (3) and Sylow's theorems that $p$ divides $|G:H|$ for every non-transitive subgroup $H<G$. By for instance \cite[\S 2 and \S 4]{EOs}, the module $\Ind^G_{H}\unit$ is zero in the semisimplification (negligible) if and only if $p$ divides $|G:H|$, which thus implies (2).

Clearly (3) implies (4). That (4) implies (3) follows from the standard fact that if a finite group $G$ acts transitively on a set $W$ of cardinality $p^n$, then so does any $p$-Sylow subgroup. For completeness we include the proof:

    Pick $x\in W$. Then the stabiliser subgroup $G_x<G$ satisfies $$|G|=|G_x||W|=rp^{m+n}$$ for some $m\ge0$ and $r$ not divisible by $p$. Now for a $p$-Sylow subgroup $P<G$ we have
    $$p^{m+n}=|P|=|P_x||P\cdot x|,$$
    where $|P_x|$ divides $p^m$ (since $P_x<G_x$) and $|P\cdot x|\le p^n$. Hence $|P_x|=p^m$ and $|P\cdot x|= p^n$, meaning $P\cdot x=W$.
\end{proof}

\begin{remark}
    Due to \Cref{SMGExists}, it is logical to focus on minimal transitive $p$-groups. Even though in the sequel we will focus on one particular family, this is still a very rich class of permutation groups. Examples include $C_p<S_p$ or, more generally, $C_{p^{n_1}}\times\dots\times C_{p^{n_t}}<S_{p^{n_1+\dots+n_t}}$, see \Cref{SecPerm} for the case $n_i=1$. There are also semidirect product groups $C_8\rtimes C_2$ or $C_2^3\rtimes C_4$ which are minimal transitive subgroups of $S_{16}$.
\end{remark}

\subsubsection{}\label{SecPerm}
For the remainder of the section, we fix $n\in\mZ_{>0}$ and set
$$W\;:=\;\{1,2,\cdots,p^n\}.$$
For every $a\in W$ we consider the $p$-adic expansion
$$a-1\;=\;\sum_{i=1}^n a_i p^{i-1}$$
with $0\le a_i<p$.
For each $1\le i\le n$, we then have a subgroup $C_p<S_{p^n}$ generated by the permutation which sends $1\le a\le p^n$ to $a'$, so that the $p$-adic expansions of $a-1$ and $a'-1$ are identical, except that $a_i$ is replaced by $a_i+1$ modulo $p$. 

We will henceforth focus our attention on the resulting class of transitive permutation $p$-groups
$$C_p^n=\prod_{i=1}^nC_p \;<\; S_{p^n}.$$

\begin{example}\label{ExPermG}
\begin{enumerate}
    \item For $G=C_p^2<S_{p^2}$, all $p+1$ non-trivial proper subgroups $H<G$ ({\it i.e.}~those subgroups isomorphic to $C_p$) are non-transitive. The corresponding modules $\Ind^G_{H}\unit$ have homological support labelled by the $p+1$ points of $\mP^1(\mF_p)\subset\mP^1(\bk)$.
    \item In case $p=2$ and $n=2$, the permutation group $C_2^{2}<S_4$ consists of the identity and the permutations
    $$(12)(34),\; (13)(24)\;\mbox{and}\; (14)(23).$$
\end{enumerate}

\end{example}

\begin{remark}
    A more intrinsic definition of the permutation group in \Cref{SecPerm} is as follows. If we let $C_p^n$ act on itself via (left or right) multiplication, then the corresponding permutation group
    $$C_p^{n}\;<\;\Perm(C_p^n)$$
    gives our set-up (after a suitable enumeration of the elements in $C_p^n$).
\end{remark}

\subsubsection{}

\Cref{SecKlein} will be devoted to classifying all $\bV$-functor for the permutation group $C_2^2<S_4$ from \Cref{ExPermG}(2), while \Cref{SecTilt} -- \Cref{Sec54} work towards a specific case for the more general $C_p^n<S_{p^n}$. Here we point out the lack of `interesting' $\bV$-functors for another class of minimal transitive $p$-groups, namely $C_{p^n}<S_{p^n}$. In particular, for $S_4$, this leaves only $C_2^2<S_4$ as interesting minimal $2$-subgroup.

\begin{lemma}\label{LemCyc}
    The only $\bV$-functors for the permutation group $C_{p^n}<S_{p^n}$ are the descendants of the semisimplification $\Sigma:\Rep C_{p^n}\to \overline{\Rep C_{p^n}}$.
\end{lemma}
\begin{proof}
    The classification of indecomposable modules in $\Rep C_{p^n}$ comprises one isomorphism class $M_i$ for every dimension $1\le i\le p^n$. If $\Theta$ is a $\bV$-functor then, by \Cref{RemHProj}, it sends every $M_{ap}=\Ind^{C_{p^n}}_{C_{p^{n-1}}}M_a$, for $1\le a\le p^{n-1}$ to zero (where we abuse notation by labelling indecomposable modules $C_{p^{n-1}}$ again by $M_{a}$). In other words, $\Theta(M_i)=0$ if and only if $p$ divides $i=\dim M_i$ if and only if $\Sigma(M_i)=0$.

    Now, set $X:=\Theta(M_2)$, which is an object in a tensor category $\cV$. By construction there appear only finitely many indecomposable summands in the tensor powers~$X^{\otimes m}$. It then follows from the previous paragraph that the growth dimension of $X$, as defined for instance in \cite[\S 4]{CEO}, must be the same as that of $\Sigma(M_2)$. By construction the latter is strictly smaller than $2=\dim M_2$. Hence $X$ is simple, and therefore $\Theta$ must send $\unit\hookrightarrow M_2$ to zero. Indeed, by duality, the alternative implies that there exist $\unit \hookrightarrow X$ and $X\tto \unit$ in $\cV$ which compose to zero.

    But if $\unit\hookrightarrow M_2$ is sent to zero by $\Theta$, then its kernel includes all morphisms $\unit\to M_i$ for $i>1$, meaning it must be the ideal of negligible morphisms.
\end{proof}

\subsection{\texorpdfstring{$\bO$}{O}-functors from \texorpdfstring{$\bV$}{V}-functors}
Recall the symmetric monoidal functor in \Cref{eqTPF} and fix $i\in\mZ_{>0}$.
\begin{theorem}\label{ThmOfromV}
    Given a permutation group $G<S_{p^i}$ and a $\bV$-functor $\Theta: \Rep G\to \cV$, the family of $\Frob^i${\rm -LSM} functors
    $$\Phi_{\cC}:\;\cC\xrightarrow{X\mapsto X^{\otimes p^i}}\cC\boxtimes \Rep S_{p^i}\xrightarrow{\cC\boxtimes \Res}\cC\boxtimes \Rep G\xrightarrow{\cC\boxtimes \Theta}\cC\boxtimes\cV$$
    yields an $\bO^i\{\cV\}$-functor.
\end{theorem}
\begin{proof}
The composite LSM functor
$$\Rep S_n\xrightarrow{\Res^{S_n}_G}\Rep G\xrightarrow{\Theta}\cV$$
sends $\Ind^{S_n}_{S_\lambda}M$ to zero, for every $\lambda\vdash n$, other than $\lambda=(n)$, and $M\in \Rep_{\bk} S_{\lambda}$.
Indeed, we can reduce this to the case $M=\bk$, which in turn follows from Mackey's theorem.

By definition, $\Phi_{\cC}$ is symmetric monoidal. Since 
$$(X\oplus Y)^{\otimes p^i}\;\simeq\; X^{\otimes p^i}\oplus Y^{\otimes p^i}\oplus \bigoplus_{j=1}^{p^i-1} \Ind^{S_{p^i}}_{S_{p^i-j}\times S_j}(X^{\otimes p^i-j}\otimes Y^{\otimes j}),$$
it follows from the previous paragraph that $\Phi_{\cC}$ is additive. Since all but the first functor in the composition of $\Phi_{\cC}$ are $\bk$-linear, it follows that $\Phi_{\cC}(\lambda f)=\lambda^{p^j}\Phi_{\cC}(f)$, for a morphism~$f$ and $\lambda\in\bk$, so $\Phi_{\cC}$ is $\Frob^i$-linear.

    That the diagrams in \Cref{DefO}(1) are commutative follows from \Cref{LemPoFu} and \Cref{Lem:Comm}. 
\end{proof}

We have an obvious compatibility between the notions of descendants of $\bV$-functors and of $\bO^i$-functors, via \Cref{ThmOfromV}.

\begin{example}\label{ExCpFr}
    The Frobenius functor $\Fr$ is obtained by the principle in \Cref{ThmOfromV} applied to the permutation group $C_p<S_p$ and the $\bV$-functor 
    $$\Sigma:\Rep C_p\to\overline{\Rep C_p}.$$
\end{example}

\begin{proposition}\label{PropExIdeal}
    Consider a permutation group $G<S_{p^i}$ and two $\bV$-functors $\Theta^j:\Rep G\to\cV_j$ with kernels $J_j$, for $j\in\{1,2\}$. Denote by $\Phi^1$, $\Phi^2$ the resulting $\bO^i\{\cV_j\}$-functors from \Cref{ThmOfromV}. Let $\cC$
be a $\Phi^2$-exact tensor category.
\begin{enumerate}
    \item If $J_1\subset J_2$,  then $\cC$ is also $\Phi^1$-exact.
    \item Assume that $\cC$ has property (P). If every object (whose identity morphism is) in $J_1$ is also in $J_2$, then $\cC$ is also $\Phi^1$-exact.
\end{enumerate}
    
\end{proposition}
\begin{proof}
    Fix a faithful exact functor $\omega:\cC\to\Vecc$, leading to commutative diagrams
    $$\xymatrix{\cC\ar[r]&\cC\boxtimes \Rep G\ar[rr]^{\cC\boxtimes \Theta^j}\ar[d]^{\omega\boxtimes(\Rep G)}&&\cC\boxtimes \cV_j\ar[d]^{\omega\boxtimes \cV_j}\\
    &\Rep G\ar[r]& (\Rep G)/J_j\ar[r]& \cV_j,
    }$$
    where the top row composes to $\Phi_{\cC}^j$ and the bottom row composes to $\Theta^j$. 

    By \Cref{Thm:FE}, $\Phi^j$-exactness is determined by faithfulness of the top row. Since the two arrows to $\cV_j$ are faithful, $\Phi^i$-exactness is thus determined by faithfulness of the composite functor
    $$\cC\to \Rep G\to(\Rep G)/J_j,$$
    where the first functor does not depend on $j$. Part (1) follows immediately and part (2) follows by adding in \Cref{ExFin}.
\end{proof}

\begin{proposition}\label{NewProp}
    Consider a $\bV$-functor $\Theta:\Rep G\to \cV$ for a permutation group $G<S_{p^i}$ and the resulting $\bO^i\{\cV\}$-functor $\Phi$. Let $\cC$ be a tensor category with (enough) projective objects. Then the following conditions are equivalent:
    \begin{enumerate}
        \item $\cC$ is $\Phi$-exact;
        \item there are projective objects $P,Q$ in $\cC$ with $\Theta(\Hom(Q,P^{\otimes p^i}))\not=0$;
        \item for each projective object $P$ in $\cC$, there is a projective object $Q$ for which $\Theta(\Hom(Q,P^{\otimes p^i}))\not=0$.
    \end{enumerate}
\end{proposition}
\begin{proof}
    This is an immediate application of \Cref{ExFin} and \Cref{RemHomPP}.
\end{proof}

\begin{example}
    A tensor category $\cC$ with projective objects is Frobenius exact if and only if there are projective objects $P,Q$ in $\cC$ for which the $S_p$-representation $\Hom(Q,P^{\otimes p})$ has a non-negligible (meaning non-projective) summand. 
\end{example}

\begin{remark}\label{RemBd}
Consider the functor
$$\cC\xrightarrow{(-)^{\otimes p^i}}\cC\boxtimes\Rep S_{p^i}\xrightarrow{\omega\boxtimes \Rep S_{p^i}}\Rep S_{p^i},$$
for a faithful exact functor $\omega:\cC\to\Vecc$. For the moment it is an open question how surjective this functor is in general, see \cite{CoNew} for some recent results. If there exists a suitable monoidal subcategory $\cR\subset \Rep S_{p^i}$ in which this functor has to take values for all $\cC$ (of moderate growth), one could introduce partially defined $\bV$-functors $\cR\to\cV$, which would still lead to $\bO$-functors.
\end{remark}

\subsection{Interlude: tilting modules}\label{SecTilt}
We denote the affine group scheme (over $\mZ$) of 2 by 2 matrices with determinant~1 by $SL_2$. So for any field $K$, we have the corresponding abstract group $SL_2(K)$. We also consider subgroups
$$U,T<B<SL_2,$$
where $B$ contains those matrices with entry $0$ above the diagonal, and $U$ additionally with entry $1$ on the diagonal, and $T$ consists of diagonal matrices.

We fix $n\in\mZ_{>0}$. When considering the elementary abelian $p$-group $C_p^n$, we denote the generators by
$$g_i=(0,0,\cdots, 1, 0,\cdots 0)\in C_p\times C_p\times \cdots \times C_p,$$
with the generator $1$ of $C_p=\mZ/(p)$ in position $i$.
\subsubsection{}
We consider elements (with $\bk^+$ the additive group of the field)
$$\ulambda=(\lambda_1,\lambda_2,\cdots, \lambda_n)\,\in\,\mA^{n}(\bk)\;\xrightarrow{1:1}\; \Hom(C_p^n,\bk^+),$$
where the bijection sends $\ulambda$ to the group homomorphism $\{g_i\mapsto \lambda_i\}$.

We denote the subset of $\mA^n(\bk)$ of elements that correspond to injective group homomorphisms by $$\mA^n_f(\bk)\;\subset\;\mA^n(\bk).$$
It is easy to give more explicit descriptions of this subset, but we do not require them presently.

For the rest of this section, we fix $\ulambda\in\mA^n_f(\bk)$.
We denote by
$$U_n<\bk^+=U(\bk)<SL_2(\bk)$$
the image of the corresponding group homomorphism $C_p^n\hookrightarrow \bk^+$.

\begin{example}
   The most canonical choice is $U_n=U(\mF_{p^n})$, which corresponds to choosing $\ulambda$ such that $\{\lambda_i\}\subset\bk$ form a basis of $\mF_{p^n}$ over $\mF_p$. 
\end{example}



\subsubsection{}Recall from \Cref{prelim::representations} that $\Rep_{\bk} SL_2$ is the category of finite-dimensional rational $SL_2(\bk)$-representations, and recall from \Cref{TiltSL2} the basic properties of $\Tilt SL_2\subset\Rep SL_2$. We will focus on the forgetful (restriction) functor
\begin{equation}\label{EqRes}
    \Rep_{\bk} SL_2\;\subset\; \Rep_{\bk} SL_2(\bk)\;\to\;\Rep_{\bk}U(\bk)\;\to \Rep_{\bk}U_n
\end{equation}
and more precisely on the composite functor
\begin{equation}\label{DefRlambda}
    R_{\ulambda}\;:\; \Tilt SL_2\hookrightarrow \Rep_{\bk} SL_2\to \Rep_{\bk}  U_n\xrightarrow{\sim} \Rep_{\bk}C_p^n
\end{equation}

We will use the same symbol for an $SL_2$-representation and for its restriction to $U_n$. For $V$ the natural $SL_2(\bk)$-representation, we set
$$V_{\ulambda}:= R_{\ulambda}(V)\;\in\;\Rep_{\bk}C_p^n.$$

 We denote by 
 $$\cD_{\ulambda}\;\subset\; \Rep C_p^n,$$
 the essential image of $R_{\ulambda}$, which is the full monoidal subcategory of $\Rep_{\bk}C_p^n$ comprising direct summands of direct sums of tensor powers of $V_{\ulambda}$. We denote by $\cK=\cK_{\ulambda}$ the tensor ideal of $\cD_{\ulambda}$ generated by any monomorphism $\unit\hookrightarrow V_{\ulambda}$ (which is unique up to a scalar).

\begin{remark}
    The representation $V_{\ulambda}$ depends, up to isomorphism, only on the image of $\ulambda$ in $\mP^{n-1}(\bk)$. Moreover, $\mP^{n-1}(\bk)$ canonically parametrises indecomposable two-dimensional $C_p^n$-representations, while the image of $\mA^n_f(\bk)\to\mP^{n-1}(\bk)$ parametrises faithful two-dimensional representations. 
\end{remark}

Recall from \Cref{TiltSL2} that $I_n$ is defined as the tensor ideal generated by the $n$-th Steinberg module $St_n=T_{p^n-1}$ in $\Tilt SL_2$.

\begin{theorem}\label{FirstResult}
\begin{enumerate}
    \item There is a faithful LSM functor $(\Tilt SL_2)/I_n\to\Stab C_p^n$ that fits into the commutative diagram
    $$\xymatrix{
    \Tilt SL_2\ar[rr]^{R_{\ulambda}}\ar[d]&& \Rep C_p^n\ar[d]\\
    (\Tilt SL_2)/I_n\ar[rr]&& \Stab C_p^n.
    }$$
    \item There is an LSM-equivalence
    $(\Tilt SL_2)/I_n\to \cD_{\ulambda}/\cK$ that fits into the commutative diagram
    $$\xymatrix{
    \Tilt SL_2\ar[rr]^{R_{\ulambda}}\ar[dd]\ar[rd]&& \Rep C_p^n\\
    &\cD_{\ulambda}\ar@{^{(}->}[ru]\ar[rd]\\
    (\Tilt SL_2)/I_n\ar[rr]^{\sim}&& \cD_{\ulambda}/\cK.
    }$$
\end{enumerate}
\end{theorem}

By slight abuse of notation, we will denote both functors constructed in \Cref{FirstResult} again by $R_{\ulambda}$.
The remainder of this subsection is devoted to the proof of \Cref{FirstResult}.
We denote by $\nabla_i$, $i\in\mN$, the Weyl modules of $SL_2$, see \cite{Jantzen}.

\begin{lemma}\label{LemNabla}
    For all $0\le i<p^n$, we have 
    \begin{enumerate}
        \item $\dim_{\bk}\Hom_{U_n}(\unit,\nabla_i)=1$,
        \item $\dim_{\bk}\Hom_{U_n}(V,\nabla_i)=\begin{cases} 1&\mbox{if $\,i=0$,}\\
        2&\mbox{if $\,i>0$.}\end{cases}$
    \end{enumerate}
    
\end{lemma}
\begin{proof}
We can realise $\nabla_i$ as the degree $i$ component of the graded algebra $\bk[x,y]$, with $x,y$ in degree 1. The action of $\lambda\in \bk\simeq U(\bk)$ is then given by
$$x^{a}y^b\;\mapsto \; (x+\lambda y)^ay^b.$$
The condition on a vector $\sum_jc_jx^{i-j}y^j$ in $\nabla_i$ to be $U_n$-invariant is thus expressed in terms of polynomials over $\bk$ of degree bounded by $i<p^n=|U_n|$ being zero when evaluated at every $\lambda\in U_n\subset \bk$. Hence all coefficients must be identically zero, yielding
\begin{equation}\label{eqFpn}\binom{i-j}{l}c_j=0,\quad\mbox{for all }\;1\le l\le i-j.\end{equation}
Picking $l=i-j$ shows that $c_j=0$ for all $j<i$, leaving indeed only a one-dimensional space of invariants, proving part (1).

The condition for a vector to be in the image of the generator of $V$ is the same as the ones in \Cref{eqFpn}, but omitting the case $ j=i-1$. This proves part~(2).
\end{proof}

\begin{remark}
    The bound $i<p^n$ in \Cref{LemNabla} is sharp. Indeed, taking $U_n:=U(\mF_{p^n})$ allows one to compute directly that parts (1) and (2) fail for $i=p^n$.
\end{remark}

For a rational $SL_2$-representation $M$ with a $\nabla$-flag (good filtration in the terminology of \cite{Jantzen}), we denote the length of this filtration by $\ell^{\nabla}(M)$. For example, $\ell^\nabla(\nabla_i)=1$. 

\begin{proposition}\label{PropT}
For all $0\le i<p^n$, we have 
    \begin{enumerate}
    \item
    $\dim_{\bk}\Hom_{U_n}(\unit,T_i)\,=\,\ell^\nabla(T_i),$
    \item $\dim_{\bk}\Hom_{U_n}(V,T_i)\,=\,2\ell^\nabla(T_i)-\dim_{\bk}\Hom_{SL_2}(\unit,T_{i})$.
    \end{enumerate}
\end{proposition}
\begin{proof}
We start with part (1). By \Cref{LemNabla}(1), we already know that
$$\dim_{\bk}\Hom_{U_n}(\unit,T_i)\,\le\,\ell^\nabla(T_i).$$
To prove the inequality in the other direction, we will prove the stronger claim
    $$\dim_{\bk}\Hom_{U}(\unit,T_i)\,\ge\,\ell^\nabla(T_i),$$
    where the left-hand side now takes into account the action of the affine group scheme~$U$. We can equivalently prove
    $$\sum_{j\in\mN}\dim_{\bk}\Hom_{B}(\unit_{-j},T_i)\,\ge\,\ell^\nabla(T_i),$$
    where $\unit_{-j}$ stands for the 1-dimensional rational $B$-representation on which $x\in\bk^\times= T(\bk)$ acts via $ x^{-j}$. The latter now follows from weight considerations. Indeed, if the lowest weight appearing in $M\in \Rep_{\bk}B$ is $-j$, then 
    $$\Hom_{B}(\unit_{-j}, M)\;=\;\Hom_T(\unit_{-j},M).$$
    Furthermore, $T_i$ has a filtration (over $SL_2$)
    $$0=M_{-1}\subset M_{0}\subset M_1\subset \cdots\subset M_i=T_i,$$
    where $M_a/M_{a-1}$ is a direct sum of copies of $\nabla_a$. Hence, the left-hand side of
    $$\dim\Hom_B(\unit_{-j},M_j)\;\le\;\dim \Hom_B(\unit_{-j},T_i)$$
    equals the number of appearances of $\nabla_{j}$ in the $\nabla$-flag of $T_i$, proving part (1).

    Now we consider part (2). We have
    $$\dim\Hom_{SL_2}(\unit, T_i)\;=\;\dim M_0.$$
     By \Cref{LemNabla}(2), we thus find
$$\dim_{\bk}\Hom_{U_n}(V,T_i)\,\le\,2\ell^\nabla(T_i)-\dim_{\bk}\Hom_{SL_2}(\unit,T_{i}).$$
We can then conclude the proof similarly to part (1), by considering various lifts of $V$ to $B$.
\end{proof}

\begin{proof}[Proof of \Cref{FirstResult}]
Denote by $\cD\subset \Rep U_n$ the essential image of the category of $SL_2$-tilting modules under the functor in \Cref{EqRes}. In other words, $\cD$ is simply $\cD_{\ulambda}$ under the isomorphism $U_n\simeq C_p^n$. We denote by $\cK$ the ideal in $\cD$ corresponding to $\cK_{\ulambda}$ in $\cD_{\ulambda}$ and write $\cE=\cD/\cK$. We will prove the statements in the theorem in their equivalent guise in terms of $U_n$ and \Cref{EqRes}, rather than $C_p^n$ and \Cref{DefRlambda}.

   The functor in \Cref{EqRes} sends the generator $St_n$ of $I_n$ to $\bk U_n$. Indeed, since $St_n$ is $\nabla_{p^n-1}$, it follows from \Cref{PropT}(1) that
   $$\Hom_{U_n}(\unit, St_{n})\;\simeq\;\bk,$$
   so the conclusion follows from $\dim_{\bk} St_n=p^n$.
   It follows that \Cref{EqRes} yields a faithful functor
    $$\Tilt SL_2/I_n\;\hookrightarrow\;\Stab U_n,$$
    already proving part (1).
The ideal $\cK$, being non-zero, contains all morphisms through projective $U_n$-modules, so we arrive indeed at an LSM functor 
\begin{equation}\label{FunE}
    (\Tilt SL_2)/I_n\;\to\; \cE.
\end{equation}
admitting the commutative diagram of part (2). It remains to prove this functor is an equivalence.

    By construction, the functor in \Cref{FunE} is essentially surjective. Since it is also additive monoidal, and the categories are rigid, it suffices to prove that
    \begin{equation}\label{EqIso}
        \Hom_{SL_2}(\unit, T_i)\;\to\; \Hom_\cE(\unit,T_i)
    \end{equation}
    is an isomorphism, for $0\le i < p^n-1$. Indeed, $T_i=0$ in $\Tilt SL_2/I_n$ for $i\ge p^n-1$, while $I_n(\unit, T_i)=0$ for $i<p^n-1$, see \cite[Proposition~3.5]{BEO} or \cite{Selecta}. So henceforth, we assume $i<p^n-1$.

    By definition, the morphism space 
    $$\cK(\unit, T_i)\;\subset\; \Hom_{U_n}(\unit, T_i)$$ is the space of all $U_n$-morphisms which factor through $V$.
    In particular, the dimension $d$ of $\cK(\unit,T_i)$ is the maximal $d$ for which we have an embedding $V^d\subset T_i$ (over~$U_n$). In particular,
    \begin{eqnarray*}
        \dim_{\bk}\cK(\unit, T_i)&=&\dim_{\bk}\Hom_{U_n}(V,T_i)-\dim_{\bk}\Hom_{U_n}(\unit,T_i)\\
        &=&\dim_{\bk}\Hom_{U_n}(\unit,T_i)-\dim_{\bk}\Hom_{SL_2}(\unit,T_i),
    \end{eqnarray*}
where the second equality follows from \Cref{PropT}. Consequently, we find
$$\dim_{\bk}\Hom_{SL_2}(\unit, T_i)\;=\;\dim_{\bk}\Hom_\cE(\unit,T_i),$$
    as required for \Cref{EqIso} to be an isomorphism. To conclude the proof, we need to observe that a non-zero $SL_2$-morphism $\unit\to T_i$, $i<p^n-1$, cannot factor through $\unit\hookrightarrow V$ over $U_n$, which is clear from the treatment in the proof of \Cref{PropT}. Indeed, images of the injective $U_n$-morphisms $V\hookrightarrow T_i$ intersect trivially with the trivial $SL_2$-subrepresentation $M_0\subset T_i$, whereas all $SL_2$-morphisms $\unit\to T_i$ land in~$M_0$. 
\end{proof}

\begin{remark}\label{RemFirst}
\begin{enumerate}
\item It is proved in \cite{BEO, AbEnv} that $\Ver_{p^n}$ is the monoidal abelian envelope of $\Tilt SL_2/I_n$. Theorem~\ref{FirstResult} thus shows that $\Ver_{p^n}$ can now also be interpreted as the monoidal abelian envelope of $\cD/\cK$, providing a definition of $\Ver_{p^n}$ purely in terms of $C_p^n$ (and an arbitrary choice of faithful two-dimensional representation).
    \item  \Cref{FirstResult} shows in particular that
    $$(\Tilt SL_2)/I_n\;\to\; \Stab C_p^n$$
    sends non-isomorphic objects to non-isomorphic objects and sends indecomposable objects to indecomposable objects. This was already observed in \cite[Proposition~4.11]{BE} (for $U_n=U(\mF_{p^n})$).
    \item Another consequence of \Cref{FirstResult} is that every two-dimensional representation of an elementary abelian $p$-group is algebraic.  Indeed, any faithful indecomposable two-dimensional representations of $C_p^n$ can be viewed as such a restriction of the vector representation of $SL_2$; and non-faithful representations correspond to faithful representations for some $C_p^s$, $s<n$.
   
\end{enumerate}
    
\end{remark}

\subsection{A conjectural \texorpdfstring{$\bO$}{O}-functor }\label{Sec54}
Fix $n\in\mZ_{>0}$, $\ulambda\in\mA_f^n(\bk)$ and set $E=C_p^n$.
In this section we give a, partly conjectural, construction of a family of $\bO_n$-functors. For $p=2=n$ the necessary conjecture can easily be proved and those functors will be studied in detail in Section~\ref{FinSec}, after the necessary preparation in Section~\ref{SecKlein}. 

More concretely, \Cref{FirstResult} allows us to introduce an LSM functor $\Rep E\supset \cD_{\ulambda}\to \Ver_{p^n}$, see \eqref{DefF} below. If we can extend this to a $\bV$-functor $\Rep E\to\Ver_{p^n}$, then we can apply our general theory to construct an $\bO_n^n$-functor. As it happens, we can describe this potential $\bO$-functor (see \Cref{Conj}(3) below) even without knowing the extension exists, at the cost of not knowing whether the result is actually monoidal. This is the approach we take below.

\subsubsection{}\label{DefTheta} We define a subfunctor
$$H_{\ulambda}\;\subset\; H^0(E,-):\Rep E\to \Vecc,\quad M\mapsto \sum_{f:V_{\ulambda}\to M}f(v),$$
where $v\in V_{\ulambda}$ spans the socle. In other words, $H_{\ulambda}(M)$ corresponds to the radical of the trace of $V_{\ulambda}$ in $M$, or most importantly, under the identification of $H^0(E,-)$ and $\Hom_{E}(\unit,-)$, we can identify $H_{\ulambda}(M)$ with the morphisms $\unit\to M$ that factor via $V_{\ulambda}$, so that
$$
H_{\ulambda}(M)\;=\; \cK_{\ulambda}(\unit,M),
\quad \text{for }M\in\cD_{\ulambda}.
$$
In particular, for projective modules $P\in \Rep E$,
$$H_{\ulambda}(P)=H^0(E,P),$$
and we can define unambiguously a bilinear functor
$$\Stab E\,\times\, \Rep E\;\to\; \Vecc,\quad (M,N)\mapsto H^0(E,M\otimes N)/H_{\ulambda}(M\otimes N).$$
\Cref{FirstResult}(1) then yields a 
bilinear functor
$$(\Tilt SL_2)/I_n\,\times\, \Rep E\;\to\; \Vecc.$$
By restricting from $(\Tilt SL_2)/I_n$ to the full subcategory $\Tilt^{[n]}SL_2$ from \Cref{TiltSL2}, using the equivalence with the category $\cP$ of projectives in $\Ver_{p^n}$ in \Cref{EquivProj}, and taking duals, we arrive at the functor
\begin{equation}\label{DefForTheta}
    \cP^{\op}\times \Rep E\;\to\; \Vecc,\quad (P,M)\mapsto H^0(E,R_{\ulambda}(P^\ast)\otimes M)/H_{\ulambda}(R_{\ulambda}(P^\ast)\otimes M),
\end{equation}
where, by abuse of notation, we use the symbol $P$ also for the corresponding $SL_2$-tilting module. For example, restricting \Cref{DefForTheta} to $M\in\cD_{\ulambda}\subset \Rep E$ yields
\begin{equation}\label{DefForTheta2}
     (P,M)\mapsto \Hom_E(R_{\ulambda}(P), M)/\cK_{\ulambda}(R_{\ulambda}(P), M)=\Hom_{\cD_{\ulambda}/\cK_{\ulambda}}(R_{\ulambda}(P),M).
\end{equation}

Now the functor in \Cref{DefForTheta} can be reinterpreted as
$$\Theta^{\ulambda}\;:\; \Rep E\;\to\; \Fun_{\bk}(\cP^{\op},\Vecc)\simeq \Ver_{p^n},$$
and by \Cref{DefForTheta2}, $\Theta^{\ulambda}$ restricted to $\cD_{\ulambda}\subset \Rep E$, yields the composite symmetric monoidal functor
\begin{equation}\label{DefF}F:\;\cD_{\ulambda}\;\to\; \cD_{\ulambda}/\cK_{\ulambda}\xrightarrow{\sim} (\Tilt SL_2)/I_n\stackrel{\Sigma^n}{\hookrightarrow}\Ver_{p^n}\end{equation}
where the first functor is just projection onto the quotient and the second functor is a quasi-inverse of the equivalence in \Cref{FirstResult}(2).

\begin{conjecture}
   \label{Conj}
    \begin{enumerate}
        \item For an arbitrary $M\in \Rep E$, the module
        $$R_{\ulambda}(St_{n-1})\otimes M$$
        belongs to $\cD_{\ulambda}$.
        \item The linear functor
        $\Theta^{\ulambda}:\Rep E\to\Ver_{p^n}$, defined in \Cref{DefTheta},
        has a symmetric monoidal structure which extends the symmetric monoidal structure of $F$ in \Cref{DefF}.
        \item The collection of $\Frob^n$-linear functors
        $$\Phi^{\ulambda}_{\cC}:\cC\to\cC\boxtimes\Ver_{p^n}\simeq\Fun_{\bk}(\cP^{\op},\cC),\quad X\mapsto F_X$$
        with
$$ F_X(P)=H^0\left(E,R_{\ulambda}(P^\ast)\otimes_{\bk} X^{\otimes p^n}\right)/ H_{\ulambda}\left(R_{\ulambda}(P^\ast)\otimes_{\bk} X^{\otimes p^n}\right)$$
forms an $\bO^n_n$-functor. Here we abuse notation in identifying projective objects in $\Ver_{p^n}$ with $SL_2$-tilting modules via equivalence~\eqref{EquivProj}.
    \end{enumerate}
\end{conjecture}

\begin{theorem}\label{ThmConj}
    We have the implications $(1)\Rightarrow (2)\Rightarrow (3)$ between the statements in \Cref{Conj}.
\end{theorem}

Before we prove the theorem, we make some comments and preparations.

\begin{remark}
    \begin{enumerate}
        \item If $p=2=n$ and $\ulambda\in\mA^2_f$, we set $\lambda:=\lambda_2/\lambda_1\in\bk\backslash\mF_2$. Then $R_{\ulambda}(St_1)$ is isomorphic to $A_1(\lambda)$ from \Cref{NotKG} and \Cref{Conj} follows immediately from \Cref{FusRules} below. More substantial evidence for the conjecture will be given in future work.
        \item Potentially, \Cref{Conj}(3) is weaker than \Cref{Conj}(1) or (2) due to \Cref{RemBd}. Indeed, to prove (3) it would be sufficient to show that (1) is true only for all $M\in \Rep E$ in the image of 
        $$\cC\to\Rep S_{p^n}\xrightarrow{\Res^{S_{p^n}}_E}\Rep E$$
        for all tensor categories $\cC$.
    \end{enumerate}
\end{remark}

In the following lemma, by a `thick tensor ideal' in $\Rep E$ we mean a collection of objects, closed under taking direct sums and summands, which is also closed under taking tensor products with arbitrary representations in $\Rep E$.

\begin{lemma}\label{LemId}
    \Cref{Conj}(1) is valid if and only if the full subcategory of $\Rep E$ comprising the image under $R_{\ulambda}$ of direct sums of copies of $T_i$ with $i\ge p^{n-1}-1$ is a thick tensor ideal in $\Rep E$.
\end{lemma}
\begin{proof}
    Clearly, the property in the lemma implies \Cref{Conj}(1). On the other hand, assume \Cref{Conj}(1) is valid. Let $T\in \Tilt SL_2$ be a direct sum of copies of $T_i,i\ge p^{n-1}-1$ and set $S=St_{n-1}$. By the ideal structure recalled in \Cref{TiltSL2}, $T$ is a summand of $S\otimes T'$ for some $T'\in \Tilt SL_2$. Moreover $S$, as a self-dual representation, is a summand of $S\otimes S\otimes S$. Hence we can write $T$ as a summand of $S\otimes S\otimes T''$ (with $T''=T'\otimes S$) 
    
    For any arbitrary $M\in \Rep E$, it follows that $R_{\ulambda}(T)\otimes M$ is a summand of
    $$ R_{\ulambda}(T''\otimes S)\otimes R_{\ulambda}(S)\otimes M\simeq R_{\ulambda}(T''\otimes S)\otimes R_{\ulambda}(T_1)\simeq R_{\ulambda}(T''\otimes T_1 \otimes S), $$
    for some $T_1\in\Tilt SL_2$, where the first isomorphism is application of \Cref{Conj}(1). Again by the ideal structure of $\Tilt SL_2$, the right-hand side is again in the full subcategory claimed to be an ideal in the lemma. 
\end{proof}

 For a morphism $u:U\to\unit$ in an LSM category, we define the morphism
$$u_\Delta:\;U\otimes U\xrightarrow{u\otimes U-U\otimes u}U.$$

\begin{lemma}\label{LemUniqueExt}
    Consider an LSM pseudo-abelian monoidal category $\widetilde{\cA}$ with full monoidal subcategory $\cA$. Consider an LSM functor $F:\cA\to\cC$ to a tensor category~$\cC$ and a morphism $f:X\to\unit$ in $\cA$ with $F(f)\not=0$. Assume that for every $Y\in\widetilde{\cA}$, the object $X\otimes Y$ belongs to $\cA$. Then there is a unique (up to isomorphism) LSM functor
$$\widetilde{F}\;:\; \widetilde{\cA}\to\cC$$
which restricts to $F$ on $\cA$. The underlying $\bk$-linear functor of $\widetilde{F}$ is given by
$$\widetilde{F}(Y)\;=\; \coker F(Y\otimes f_\Delta).$$
\end{lemma}
\begin{proof}
Observe first that, for any morphism $u:U\to\unit$ in a tensor category, the sequence
$$U\otimes U\xrightarrow{u_\Delta} U\xrightarrow{u}\unit\to 0$$
is (right) exact, see \cite[Theorem~4.2.2(ii)]{Top}. In particular, with the $\bk$-linear functor $\widetilde{F}$ as defined in the lemma, it follows that for $X\in\cA$, we have natural isomorphisms
$$\widetilde{F}(X)\;\simeq\; \coker(F(X)\otimes F(f)_{\Delta})\;\simeq\; F(X),$$
where we also used exactness of the tensor product in $\cC$.

We can give a monoidal structure on $\tilde{F}$ as follows. First we can observe that for any $A,Y\in \widetilde{\cA}$, the above displayed isomorphisms and the fact that $A\otimes X\otimes Y\simeq (A\otimes Y)\otimes X\in \cA$, yield a natural isomorphism
$$ F(A\otimes X)\otimes \widetilde{F}(Y)\;\simeq \;\coker F(A\otimes X\otimes Y\otimes f_\Delta)\;\simeq\;F(A\otimes X\otimes Y).$$

    For $Y_1,Y_2\in \widetilde{\cA}$, we thus obtain natural isomorphisms
    $$\widetilde{F}(Y_1)\otimes \widetilde{F}(Y_2)\;\simeq\; \coker\left(F(Y_1\otimes f_\Delta)\otimes \widetilde{F}(Y_2)\right)\;\simeq\; \coker F(Y_1\otimes f_\Delta\otimes Y_2),$$
    and thus by using the braiding
    $$\widetilde{F}(Y_1)\otimes \widetilde{F}(Y_2)\;\simeq\; \coker\left(F(Y_1\otimes Y_2\otimes f_\Delta )\right)\;=\; \widetilde{F}(Y_1\otimes Y_2).$$
    It can be verified that these isomorphisms equip $\widetilde{F}$ with a monoidal structure which extends the one on $F$. Uniqueness is verified using similar arguments.  
\end{proof}

\begin{proof}[Proof of \Cref{ThmConj}]
    First we prove that (1) implies (2). We start with an application of \Cref{LemUniqueExt}. We set $\widetilde{\cA}=\Rep E$ and $\cA=\cD_{\ulambda}$. We have a non-zero morphism in $\Tilt SL_2$
    $$f': T_{2p^{n-1}-2}\;\to\; \unit,$$
see \cite[Lemma~5.3.3]{Selecta}, so we set $f=R_{\ulambda}(f')$ and $X=R_{\ulambda}(T_{2p^{n-1}-2})$. Finally, we consider the composite LSM functor $F:\cD_{\ulambda}\to\Ver_{p^n}$ from \Cref{DefF}. Note that $F(f)\not=0$, since by construction the image of $f$ under the composite of the first two functors in \Cref{DefF} is simply $f'$, which is not contained in $I_n$.
Hence, by \Cref{LemId}, under \Cref{Conj}(1), all assumptions in \Cref{LemUniqueExt} are satisfied and we have the unique LSM extension
$$\widetilde{F}:\; \Rep E\to \Ver_{p^n}.$$    

Now we show that $\widetilde{F}$ is isomorphic to $\Theta^{\ulambda}$. For this we observe that for a projective object $P\in\cP$ in $\Ver_{p^n}$, which is thus of the form $\Sigma^n(T)$ for $T\in \Tilt^{[n]}SL_2$, and $M\in \Rep E$, we have natural isomorphisms
\begin{eqnarray*}
    \Hom_{\Ver_{p^n}}(P,\widetilde{F}(M))&\simeq& \Hom_{\Ver_{p^n}}(\unit, P^\ast\otimes \widetilde{F}(M))\\
    &\simeq&  \Hom_{\Ver_{p^n}}(\unit, F(R_{\ulambda}(T^\ast))\otimes \widetilde{F}(M))\\
    &\simeq&  \Hom_{\Ver_{p^n}}(\unit, F(R_{\ulambda}(T^\ast)\otimes M))\\
    &\simeq& \Hom_{E}(\unit, R_{\ulambda}(T^\ast)\otimes M)/\cK(\unit,R_{\ulambda}(T^\ast)\otimes M)\\
    &\simeq&  \Hom_{\Ver_{p^n}}(P,\Theta^{\ulambda}(M)).
\end{eqnarray*}
Here, the first isomorphism is adjunction, the second is an application of \Cref{DefF}, the third uses \Cref{Conj}(1) and the fact that $\widetilde{F}$ is monoidal and extends $F$, the fourth is an application of the equivalence in \Cref{FirstResult}(2), while the last is the interpretation of $\Theta^{\ulambda}$ as being defined by \Cref{DefForTheta}.

Now we prove that (2) implies (3) in \Cref{Conj}. Firstly we observe that $\Theta^{\ulambda}:\Rep E\to \Ver_{p^n}$ is a $\bV$-functor. We thus need to show that $\Theta^{\ulambda}(\Ind^E_H\unit)=0$ for all proper subgroups $H<E$. Chasing through the definition of $\Theta^{\ulambda}$ (and using the ideal structure in $\Tilt SL_2$) it follows that it is sufficient to show that 
$$R_{\ulambda}(St_{n-1})\otimes \Ind^E_H\unit\;\simeq\;\Ind^E_H\Res^E_HR_{\ulambda}(St_{n-1})$$
is projective. This follows since $\Res^E_HR_{\ulambda}(St_{n-1})$ is projective for any proper subgroup $H<E$. Indeed, $\Res^E_HR_{\ulambda}(St_{n-1})$ is simply $R_{\ulambda'}(St_{n-1})$ for some $\ulambda'\in\mA^{m}_f(\bk)$ with $m<n$, so the conclusion follows from \Cref{FirstResult}(1).
Now the functor $\Phi^{\ulambda}_{\cC}$ is constructed from $\Theta^{\ulambda}$ by the procedure in \Cref{ThmOfromV}, concluding the proof.
\end{proof}

\begin{example}
    The definition of $\Phi^n_{\cC}$ for $n=1$ can be seen to be equivalent with the definition of $\Fr$ in \cite[\S 4.3]{Tann}.
\end{example}


\section{The Klein 4-group}\label{SecKlein}
In this section we assume $\chara(\bk)=2$. We gather and prove some results that are needed to establish \Cref{Conj} for the case $p=2=n$, and more generally explore the full potential for $\bO$-functors based on the transitive permutation group $C_2^2<S_4$ using the formalism from \Cref{SecConst}.

\subsection{Classifying LSM functors to tensor categories}
\subsubsection{Notation}\label{NotKG}

Recall, for instance from \cite{Conlon}, that the indecomposable modules of $\bk C_2^2$ are classified as follows. For every $i\in\mZ$, we have the $i$-th Heller shift $\Omega^{i}\unit$. For instance, $\Omega^{-1}\unit$ is three-dimensional with simple top. Besides the indecomposable projective module, this leaves the (self-dual) indecomposable modules $A_m(\lambda)$ with $m\in\mZ_{>0}$ and $\lambda\in\mP^1(\bk)$. For $\lambda\in\bk=\mA^1(\bk)\subset\mP^1(\bk)$, $A_m(\lambda)$ is the $2m$-dimensional module on which the generators act as
\begin{equation}\label{DefAlambda}
    g_1\,\mapsto\,\left(\begin{array}{c|c}
       \mI_m  & 0 \\
       \hline
       \mI_m  & \mI_m
    \end{array}\right) \quad\mbox{and}\quad g_2\,\mapsto\, \left(\begin{array}{c|c}
       \mI_m  & 0 \\
          \hline
       \mJ_\lambda  & \mI_m
    \end{array}\right),
\end{equation}
where $\mI_m$ is the $m\times m$ identity matrix and $\mJ_\lambda$ is an indecomposable Jordan block with generalised eigenvalue $\lambda$. Finally, $A_m(\infty)$ is $A_m(0)$ but with the roles of the two generators reversed.

It will be useful for the sequel to fix, for every $\mu\in\mP^1(\bk)$, a morphism $f_\mu$ in $\Rep C_2^2$, uniquely determined up to a (non-zero) scalar, which yields a short exact sequence
\begin{equation}\label{seslambda}0\to \unit\xrightarrow{f_\mu} \Omega^{-1}\unit\to A_1(\mu)\to 0.\end{equation}
Up to duality, this realises $A_1(\mu)$ as a Carlson module. 

Finally, we denote by $\alpha_2$ the Frobenius kernel of the additive group $\mG_a$ and by $\Vecc_{\mZ}$ the tensor category of $\mZ$-graded vector spaces, with ordinary braiding (making the forgetful functor to $\Vecc$ symmetric monoidal).

\begin{theorem}\label{ThmClass}
    Every LSM functor from $\Rep C_2^2$ to a tensor category is a composite of a tensor functor with (precisely one) of the following LSM functors.
    \begin{enumerate}
    \item[(0)] The identity functor of $\Rep C_2^2$.
        \item For every $\lambda\in\mP^1(\bk)\backslash \mP^1(\mF_2)$ a non-faithful LSM functor 
        $$\Theta=\Theta^\lambda_1\;:\;\Rep C_2^2\;\to\; \Ver_4,$$
        uniquely determined by the property $\Theta(A_1(\lambda))\not=0$ (which is equivalent to the property $\Theta(A_2(\lambda))\not=0$).
        
        \item For every $\lambda\in\mP^1(\mF_2)$ a non-faithful LSM functor
        $$\Theta=\Theta^\lambda_0\;:\;\Rep C_2^2\;\to\; \Rep C_2,$$
        uniquely determined by the property that $\Theta(A_{1}(\lambda))$ is projective.
        \item For every $\lambda\in\mP^1(\bk)$ and $m\in\mZ_{>0}$, with $m>1$ in case $\lambda\not\in\mP^1(\mF_2)$, an LSM functor
        $$\Theta=\Theta^\lambda_m\;:\;\Rep C_2^2\;\to\; \Rep \alpha_2,$$
        uniquely determined by the properties that $\Theta(A_m(\lambda))=0$ and that $\Theta(A_{m+1}(\lambda))$ is projective.
        \item For every $\lambda\in\mP^1(\bk)$ an LSM functor
        $$\Theta=\Theta_\infty^\lambda\;:\;\Rep C_2^2\;\to\; \Vecc,$$
        uniquely determined by the properties $\Theta(A_m(\mu))=0$ for all $\mu\in\mP^1(\bk)$ and $m\in\mZ_{>0}$, and $\Theta(f_\lambda)=0$ while $\Theta(f_\mu)\not=0$ for some $\mu$.
        \item An LSM functor
        $$\Theta=\Theta_\infty=\Sigma\;:\;\Rep C_2^2\;\to\; \Vecc_{\mZ},$$
        uniquely determined by the properties $\Theta(f_\lambda)=0$ and $\Theta(A_m(\lambda))=0$ for all $\lambda\in\mP^1(\bk)$ and $m\in\mZ_{>0}$, and $\Theta(\Omega^1\unit)\simeq\unit\langle 1\rangle$, the tensor unit in degree $1$.
    \end{enumerate}
\end{theorem}

The proof of this theorem will occupy the rest of this section. First we observe that, since the modules of the form $\Ind^{C_2^2}_H\unit$, for proper subgroups $H<C_2^2$, are $A_1(0)$, $A_1(1)$, $A_1(\infty)$ and $kC_2^2$, we obtain the following corollary. 

\begin{corollary}\label{CorClass}
    Let $C_2^2<S_4$ be the permutation group from \Cref{ExPermG}(2). Every $\bV$-functor for this permutation group is a descendant (in the sense of \Cref{DefV}) of one of the functors
    $$\{\Theta^\lambda_m\mid \lambda\in\mP^1(\bk), m\in\mZ_{>0}\cup\{\infty\}\}\cup\{\Theta_\infty\}.$$
    \end{corollary}

    \begin{remark}
    \begin{enumerate}
        \item It follows from the tensor product rules in \Cref{FusRules} below that $\Theta^\lambda_m(A_n(\mu))=0$ when $\mu\not=\lambda$ and, for $\lambda\not\in\mP^1(\mF_2)$, that
    $$\Theta_1^\lambda(A_1(\lambda))\;\simeq\; V\in \Ver_4.$$
    \item It also follows from \Cref{FusRules} that $\Theta_m^\lambda(A_n(\lambda))\not=0$ if $n>m$.
    \end{enumerate}
    
\end{remark}

\begin{remark}
A neat pattern, for which the only evidence at the moment is \Cref{LemCyc}
 and \Cref{ThmClass}, would be that the only LSM functors from $\Rep C_p^r$ to incompressible tensor categories land in $\Ver_{p^r}$. 
\end{remark}

\subsection{Preparatory results}

The following result can be found on page~82 of \cite{Conlon}.

\begin{lemma}\label{FusRules}
With $\lambda,\mu\in\mP^1(\bk)$, we have the following tensor products in $\Stab C_2^2$:
\begin{enumerate}
    \item For $m,n\in\mZ_{>0}$, with $m+n>2$, 
    $$A_m(\lambda)\otimes A_n(\mu)\;\simeq\;\begin{cases}
    0,&\mbox{if $\lambda\not=\mu$;}\\
    A_{\min(m,n)}(\lambda)^2 ,&\mbox{if $\lambda=\mu$.}
    \end{cases}$$
    \item
    $$A_1(\lambda)\otimes A_1(\mu)\;\simeq\;\begin{cases}
    0,&\mbox{if $\lambda\not=\mu$;}\\
    A_{2}(\lambda),&\mbox{if $\lambda=\mu\not\in\mP^1(\mF_2)$;}\\
    A_{1}(\lambda)^2,&\mbox{if $\lambda=\mu\in\mP^1(\mF_2)$.}
    \end{cases}$$
    \item For $i,j\in\mZ$,
    $\Omega^i\unit \otimes \Omega^j\unit\;\simeq\; \Omega^{i+j}\unit.$
    \item For $i\in\mZ$, $m\in\mZ_{>0}$,
    $A_m(\lambda)\otimes\Omega^i\unit\;\simeq\; A_m(\lambda)$.
\end{enumerate}
    
\end{lemma}

\begin{lemma}\label{LemInc}
Fix $\lambda\in\mP^1(\bk)$ and $m\in\mZ_{>0}$.
\begin{enumerate}
    \item  For $1\le i\le m$, there is a unique submodule, and a unique quotient, of $A_m(\lambda)$ isomorphic to $A_i(\lambda)$. 
    \item The $\End_{C^2_2}(A_m(\lambda))$-module $\Hom_{C_2^2}(\unit,A_m(\lambda))$ has one-dimensional top, and its unique maximal submodule consists of the morphisms that factor via $A_{m-1}(\lambda)\subset A_m(\lambda)$.
    \item There is a $\bk$-basis of $\End_{C^2_2}(A_m(\lambda))$ of the following form: one composition $A_m(\lambda)\tto A_i(\lambda)\hookrightarrow A_m(\lambda)$ for each $1\le i\le m$ and $m^2$ compositions $A_m(\lambda)\tto \unit\hookrightarrow A_m(\lambda)$
\end{enumerate}
    
\end{lemma}
\begin{proof}
    Part (1) follows from the description of $A_m(\lambda)$ in \Cref{DefAlambda}. 

    For part (2), set $R=\End_{C^2_2}(A_m(\lambda))$ and $T=\Hom_{C_2^2}(\unit,A_m(\lambda))$. Since $R$ is a local algebra, finite-dimensional over an algebraically closed field, it suffices to show that all vectors that are not in the subspace
        $$T_1:=\Hom_{C_2^2}(\unit, A_{m-1}(\lambda))\;\subset\; \Hom_{C_2^2}(\unit, A_{m}(\lambda))=T,$$
    defined via $A_{m-1}(\lambda)\subset A_m(\lambda)$ from part (1),
    generate $T$ as an $R$-module. Note that $\dim T=m$ and $\dim T_1=m-1$. For $v\in T\backslash T_1$, we consider the composite
    $$\unit\xrightarrow{v} A_m(\lambda)\twoheadrightarrow A_{m-1}(\lambda)\hookrightarrow A_m(\lambda),$$
    where the second and the third morphism come from part (1) and its dual. Clearly, this composite is in $T_1$, but it also can be seen that it is {\em not} included in the codimension $2$ subspace $\Hom_{C_2^2}(\unit,A_{m-2}(\lambda))$ of $T$. Continuing this procedure shows that $Rv=T$.

    For part (3), the case $m=1$ is straightforward, so we assume $m>1$. \Cref{FusRules}(1) implies that 
    $$A_m(\lambda)^{\otimes 2}\;\simeq\; A_m(\lambda)^2\oplus (\bk C_2^2)^{m^2-m}.$$
    Hence, by adjunction, the proposed basis has the correct size, and it suffices to show that we can choose a linearly independent set of such morphisms, which is obvious by part (1).
\end{proof}

The following lemma is trivial, but it will be useful to have it written out.
\begin{lemma}\label{SillyLemma}
   For a self-dual object $Y$ satisfying $Y^{\otimes 2}\simeq Y^2$ in an additive monoidal category, we have canonical isomorphisms of vector spaces
    $$\End(Y)\simeq\Hom(\unit, Y)^2\simeq \Hom(Y,\unit)^2.$$
\end{lemma}

\begin{lemma}\label{LemTech}
    Consider $A_m(\lambda)\in \Rep C_2^2$ with $\lambda\in\mP^1(\bk)$ and $m\in\mZ_{>0}$, excluding the case $m=1$ for $\lambda\not\in\mP^1(\mF_2)$. Denote by~$\cB=\cB_m(\lambda)$ the additive (rigid) monoidal subcategory of $\Stab C_2^2$ generated by~$A_m(\lambda)$. 
\begin{enumerate}
    \item There is precisely one tensor ideal $\cI$ in $\cB$ that does not contain $A_m(\lambda)$ and that is the kernel of an LSM functor to a tensor category.
    \item The ideal $\cI$ in part (1) is the maximal tensor ideal in $\cB$ that does not contain $A_m(\lambda)$. It does not contain all morphisms $A_m(\lambda)\to\unit$.
    \item The symmetric monoidal category $\cB/\cI$ is a tensor category, concretely
    $$\cB/\cI\;\simeq\;\begin{cases}
        \Rep C_2,&\mbox{ if $m=1$ (and $\lambda\in\mP^1(\mF_2)$)},\\
        \Ver_4^+,&\mbox{ if $m=2$ and $\lambda\not\in\mP^1(\mF_2)$},\\
        \Rep \alpha_2,&\mbox{ otherwise}.
    \end{cases}$$
    
\end{enumerate}
\end{lemma}
\begin{proof} Set $X=A_m(\lambda)$. By \Cref{FusRules}(1) and (2), in $\cB\subset\Stab C_2^2$ we have
\begin{equation}\label{XXXX}
    X\otimes X\;\simeq\; X\oplus X.
\end{equation} In particular, the objects in $\cB\subset \Stab C_2^2$ are direct sums of copies of $\unit$ and $X$.
    
Let $Y$ be a non-zero image of $X$ in some tensor category $\cV$ under an LSM functor. By \Cref{SillyLemma}, $Y$ admits non-zero morphisms to and from $\unit$. In particular, in the Grothendieck ring $K_0(\cV)$, we have $[Y]=[Z]+[\unit]$ for some $Z\in\cV$.
 \Cref{XXXX} shows that in $K_0(\cV)$:
 $$([Z]+1)^2\;=\; 2[Z]+2\quad\Rightarrow\quad [Z]^2=1,$$ so that $Y$ must either be a self-extension (possibly split) of $\unit$, or a direct sum of $\unit$ with an invertible object.

The insight from the previous paragraph shows that for any tensor ideal $J$ for which $\cB/J$ admits a faithful LSM functor to a tensor category $\cV$, we must have either
\begin{equation}
    \label{eqdim1} \dim \Hom_\cB(\unit,X)\;=\; 1+\dim J(\unit,X),
\end{equation}
or alternatively (and forcing $Y\simeq \unit^2$) 
\begin{equation}
    \label{eqdim2}\dim \Hom_\cB(\unit,X)\;=\; 2+\dim J(\unit,X).\end{equation} 
    In the case of \Cref{eqdim2}, the functor $\cB/J\to\cV$ is fully faithful, for instance by \Cref{SillyLemma},
    so that the algebra $\End_{\cB/J}(X)\simeq\End_{\cV}(\unit^2)$ must be the matrix algebra $\mathrm{Mat}_2(\bk)$, which is impossible since $\End_{\cB}(X)$ is local.

  By a direct application of \cite[Proposition~3.2.1]{Selecta}, we have an inclusion preserving bijection between proper tensor ideals in $\cB$ and $\End_{\cB}(X)$-submodules of $M:=\Hom_\cB(\unit,X)$. We have
    $$M\simeq\Hom_{C_2^2}(\unit, A_m(\lambda)),$$
    and \Cref{LemInc}(2) thus implies that there is a unique tensor ideal $J$ in $\cB$ for which \Cref{eqdim1} is satisfied.
Moreover, this is the maximal tensor ideal which does not contain all morphisms $\unit\to X$, which, by \Cref{SillyLemma}, is the maximal tensor ideal that does not contain $X$. We call this ideal $\cJ$.

We have thus far proved uniqueness of the ideal in (1), and that this hypothetical ideal would have to satisfy the two properties in (2). To conclude the proof we demonstrate that (3) holds for the ideal $\cJ$, which then implies (1) too. 

    In $\cB/\cJ$, we have unique (up to scalars) non-zero morphisms $\alpha:\unit\to X$ and $\beta:X\to \unit$. Moreover, $\End_{\cB/\cJ}(X)$ is 2-dimensional by \Cref{XXXX} and \Cref{SillyLemma}. We can observe that $\beta\circ\alpha=0$ in $\cB/\cJ$, as that is already the case in $\Rep C_2^2$. Furthermore, we claim $\alpha\circ\beta\not=0$. Indeed, we can consider the basis in \Cref{LemInc}(3) and claim that all endomorphisms $A_m(\lambda)\tto A_i(\lambda)\hookrightarrow A_m(\lambda)$, for $i<m$ are zero in $\cB/\cJ$. If not, then by adjunction and \Cref{FusRules} there would be a composite morphism $\unit\to A_i(\lambda)^2\to A_m(\lambda)^2$ which is not in $\cJ$. By \Cref{LemInc}(2), such morphisms $\unit\to A_m(\lambda)$ are in $\cJ$. Hence, some composite morphism $A_m(\lambda)\tto \unit\hookrightarrow A_m(\lambda)$ must not be sent to zero in $\cB/\cJ$, which implies $\alpha\circ\beta\not=0$.
    
   It thus follows that $\cB/\cJ$ is $\bk$-linearly equivalent with $\bk[x]/x^2\mbox{-mod}$, which implies that $\cB/\cJ$ is abelian and hence a tensor category. By \cite{EG}, $\cB/\cJ$ admits a tensor functor to $\Ver_4^+$. It follows that $\cB/\cJ$ is either tannakian (corresponding to an affine group scheme $G$ with $\bk[G]=\bk[x]/x^2$) or $\Ver_4^+$, which results in the three possibilities in (3). 

   To decide which of the three tensor categories we obtain for every given pair $(\lambda,m)$, we can use \Cref{Rem3Cases} below, which reduces the problem to a tedious computation which we omit. A considerable amount can actually be avoided. Indeed, the case $\lambda\in\mP^1(\mF_2)$, $m=1$ follows immediately from three inclusions $\Rep C_2\subset \Stab C_2^2$ coming from the three epimorphisms $C_2^2\tto C_2$. The case $\lambda\not\in\mP^1(\mF_2)$, $m=2$ follows from \Cref{FirstResult}(2) and \cite[Remark~3.10]{BE}. For the remaining cases, one can use the functors in \Cref{F4case}
and \Cref{F2case} to demonstrate that $\cB/\cI$ must be tannakian, eliminating the possibility $\Ver_4^+$.
\end{proof}

\begin{remark}\label{Rem3Cases}
    The categories $\Rep C_2$, $\Rep \alpha_2$ and $\Ver_4^+$ are given by $\bk[x]/x^2\mbox{-mod}$ as $\bk$-linear categories. Furthermore, for the $P$ the regular module, we have in all three categories
    $$P\otimes P\;\simeq\; P\oplus P.$$
    A convenient way to tell the three tensor categories apart is by considering the action of $1+\sigma \in \bk S_2$ (with $\sigma$ the generator of $S_2$) via the braiding on $P^{\otimes 2}$. The action is (up to conjugation by $GL_2(\bk[x]/x^2)$) given by the following elements of $\mathrm{Mat}_2(\bk[x]/x^2)\simeq\End(P^2)$
    $$\left(\begin{array}{cc}
        x & 0 \\
        0 & 0
    \end{array}\right),\quad \left(\begin{array}{cc}
        0 & 0 \\
        x & 0
    \end{array}\right),\quad\mbox{and}\quad \left(\begin{array}{cc}
        0 & x \\
        x & 0
    \end{array}\right)$$
    in $\Rep C_2$, $\Rep \alpha_2$ and $\Ver_4^+$ respectively.
\end{remark}

\subsection{Proof of \texorpdfstring{\Cref{ThmClass}}{Theorem ...}}

Let $\Theta:\Rep C_2^2\to \cV$ be an LSM  functor to a tensor category $\cV$. If $\Theta$ is faithful, then it is exact by \Cref{Thm:FE} and therefore a tensor functor. Faithful functors thus correspond to those in \Cref{ThmClass}(0) and henceforth we can restrict to non-faithful functors. Since the projective modules are contained in each tensor ideal, we are thus interested in LSM functors
$$\Theta:\;\Stab C_2^2\;\to\;\cV.$$
We start by considering the following 4 options \ref{F4case} -- \ref{subsubinfty}.

\subsubsection{Assume $\Theta(A_2(\lambda))\not=0$ for some $\lambda\in\mP^1(\bk)\backslash \mP^1(\mF_2)$}\label{F4case}

Denote by $\cA$ the full additive monoidal subcategory of $\Stab C_2^2$ generated by $A_1(\lambda)$. By \Cref{FusRules}, 
$$\cA\;\subset\;\widetilde{\cA}\;:=\;\Stab C_2^2$$
comprises direct sums of $\unit$, $A_1(\lambda)$ and $A_2(\lambda)$ and satisfies the hypothesis in \Cref{LemUniqueExt}, for $X:=A_2(\lambda)$.

Furthermore, by \Cref{SillyLemma}, there must be a morphism $A_2(\lambda)\to\unit$ which is not sent to zero by $\Theta$. By \Cref{LemUniqueExt}, restriction yields a correspondence between LSM functors $\Stab C_2^2\to\cV$ that do not kill $A_2(\lambda)$ and LSM functors $\cA\to\cV$ with the same condition.

Moreover, if we set $B:=\Theta(A_1(\lambda))\in\cV$, then \Cref{FusRules} implies
$$B^{\otimes 3}\;\simeq\; B\oplus B.$$
Such objects $B$ cannot contain $\unit$ as a subobject. In fact, they must be simple. Consequently, $\Theta$ must send all morphisms $\unit\to A_1(\lambda)$ to zero.

Setting $\ulambda=(1:\lambda)$, we can observe that $\cA$ is precisely the image of $\cD_{\ulambda}$ from \Cref{SecTilt} in the stable category.
By the previous paragraph, LSM functors $\cA\to\cV$ thus correspond to LSM functors $\cD_{\ulambda}/\cK\to \cV$, with $\cK$ also as in \Cref{SecTilt}. Applying \Cref{FirstResult}(2) yields
\begin{equation}\label{eqVer4}\cD_{\ulambda}/\cK\;\simeq\; (\Tilt SL_2)/I_2\;\simeq\; \Ver_4,\end{equation}
where the last equivalence is in \cite[Remark~3.10]{BE}. This equivalence exchanges identifies $A_2(\lambda)$ and the projective cover of $\unit$ in $\Ver_4$.

Since $A_2(\lambda)$ is projective in the tensor category $\Ver_4$, \Cref{ExFin} shows that all LSM functors $\cD_{\ulambda}/\cK\to \cV$ that do not kill $A_2(\lambda)$ are (faithful) tensor functors.
In conclusion, all non-faithful LSM functors $\Theta$ from $\Rep C_2^2$ to tensor categories satisfying $\Theta(A_2(\lambda))\not=0$ factor through a unique functor as in \Cref{ThmClass}(1).


\subsubsection{$\Theta(A_1(\lambda))\not=0$ for some $\lambda\in\mP^1(\mF_2)$}\label{F2case}

We freely use \Cref{FusRules}.
Observe that non-zero morphisms $A_1(\lambda)\to\unit$ are not annihilated by $\Theta$, due to \Cref{SillyLemma}.
By \Cref{LemUniqueExt}, LSM functors $\Stab C_2^2\to\cV$ which do not kill $A_1(\lambda)$ are thus in natural bijection with LSM functors $\cA\to\cV$ that do not kill $A_1(\lambda)$, where $\cA\simeq \Rep C_2$ is the monoidal subcategory of $\Stab C_2^2$ generated by $A_1(\lambda)$. It then follows from \Cref{ExFin} that such functors factor through a unique functor as in \Cref{ThmClass}(3).

Note that these three (for $\lambda\in\mP^1(\mF_2)$) functors are of the form
$$\Rep C_2^2\simeq(\Rep C_2)\boxtimes(\Rep C_2)\xrightarrow{(\Rep C_2)\boxtimes\Sigma} (\Rep C_2)\boxtimes\Vecc\simeq \Rep C_2,$$
where we carried out semisimplification of one of the three subgroups $C_2<C_2^2$.

\subsubsection{$\Theta(A_m(\lambda))=0$ and $\Theta(A_{m+1}(\lambda))\not=0$ for some $\lambda\in\mP^1(\bk)$, $m\in\mZ_{>0}$}

Firstly, we can observe that the two conditions become contradictory for $m=1$ and $\lambda\not\in\mP^1(\mF_2)$, by \Cref{FusRules}(2).

For the remainder of this subsection we focus on the case $\lambda\not\in\mP^1(\mF_2)$ (and $m>1$), as the case $\lambda\in\mP^1(\mF_2)$ can similarly be derived from \Cref{F2case}. 
We claim that any $\Theta$ with the stated properties factors uniquely through one such LSM functor $\Stab C_2^2\to\Rep \alpha_2$.

We set $l=m+1$ for convenience. Consider the monoidal subcategory $\cB=\cB_l(\lambda)\subset\Stab C_2^2$ from \Cref{LemTech} with indecomposable objects $\unit$ and $A_l(\lambda)$. Let $J$ denote the ideal in $\cB$ which is the kernel of
$$\cB\hookrightarrow \Stab C_2^2\tto(\Rep C_2^2)/I$$
with $I$ the tensor ideal in $\Rep C_2^2$ generated by $A_{l-1}(\lambda)$.

As in \Cref{F4case} and \Cref{F2case}, it follows from \Cref{LemUniqueExt}, applied to $\widetilde{\cA}=(\Rep C_2^2)/I$ and $\cA=\cB/J$ that it suffices to show that there is a unique LSM functor $\cB/J\to \Rep\alpha_2$ such that all LSM functors $\cB/J\to\cV$ to tensor categories $\cV$ that do not annihilate $A_l(\lambda)$ factor via the former.

Now, since the ideal $I$ does not contain $A_{l}(\lambda)$, we find $J\subset\cI$, for $\cI$ the unique ideal in $\cB$ from \Cref{LemTech}. By this lemma (and the fact that $\Rep\alpha_2$ does not have non-trivial auto-equivalences), it now suffices to observe that the tensor category $\cB/\cI$ is $\Rep\alpha_2$, by \Cref{LemTech}(3).



\subsubsection{$\Theta(A_n(\mu))=0$ for all $n\ge1$ and $\mu\in\mP^1(\bk)$}\label{subsubinfty}
  The only remaining indecomposable objects in $\Stab C_2^2$ are $\Omega^i\unit$ for $i\in\mZ$. These are invertible, see \Cref{FusRules}(3).
    Every morphism $\unit\to\Omega^i\unit$, for $i>0$, must be sent to zero by~$\Theta$, since they factor as $\unit\to A_i(0)\to \Omega^i\unit$.

    Next assume that all morphisms $\unit\to\Omega^{-1}\unit$ are sent to zero by $\Theta$. Since all morphisms $\unit\to\Omega^{-i}\unit$, for $i>1$, are linear combinations of composites $\unit\to\Omega^{-1}\unit\to\Omega^{-i}\unit$, they all go to zero, too. It thus follows that $\Theta$ kills all morphisms between indecomposables that are not isomorphisms, and $\Theta$ factors through the semisimplification
    $\overline{\Rep C_2^2}$, which is the category of $\mZ$-graded vector spaces. This leads to composite functors with the one in \Cref{ThmClass}(5).

    Since the space of morphisms $\unit\to\Omega^{-1}\unit$ is two-dimensional, and must be sent by $\Theta$ to a morphism space between invertible objects in a tensor category, the alternative is that the space of morphisms $\unit\to\Omega^{-1}\unit$ that are sent to zero by $\Theta$ is spanned by $f_\lambda$ for some $\lambda\in\mP^1(\bk)$. Let $J$ denote the tensor ideal in $\Rep C_2^2$ generated by all $A_m(\mu)$, all morphisms $\unit\to \Omega^i\unit$, $i>0$, and $f_\lambda$. It follows that all morphism spaces from $\unit$ to indecomposable objects in $\cA:=(\Rep C_2^2)/J$ are zero, except that
    $$\cA(\unit, \Omega^{-i}\unit)\;\simeq\;\bk,\quad\mbox{for }\;i\ge 0.$$
    Any LSM functor $\cA\to\cV$ to a tensor category $\cV$ that does not kill all morphisms $\unit\to\Omega^{-1}\unit$ lands in $\Vecc\subset\cV$, and there is a unique faithful LSM functor $\cA\to\Vecc$. Hence, the assumptions in the current paragraph lead to functors as in \Cref{ThmClass}(4).


\subsubsection{Conclusion}
Now let $\Theta$ be any LSM functor from $\Stab C_2^2$ to a tensor category. It follows from \Cref{FusRules}(1) that $\Theta(A_m(\mu))=0$ for all $\mu$ or for all $\mu\not=\lambda$ for one $\lambda\in\mP^1(\bk)$. Moreover, if $\Theta(A_{m+1}(\lambda))=0$ then $\Theta(A_m(\lambda))=0$, by \Cref{FusRules}(1). We must therefore be in one of the cases \ref{F4case} -- \ref{subsubinfty}. This completes the proof of \Cref{ThmClass}.

\smallskip

We conclude this section with the following consequences of the proof of \Cref{ThmClass}. We denote by $\cF$ the $\bk$-linear symmetric rigid monoidal category of finite-dimensional $\mZ$-filtered vector spaces, with filtration respecting morphisms. Below we consider the LSM functors to $\Vecc_{\mZ}$ and $\Vecc$ which associate to a filtered vector space its associated graded vector space or forget the filtration, respectively.
\begin{lemma}\label{PropFilt}
    For every $\lambda\in\mP^1(\bk)$, there exists an LSM functor $F_\lambda$, admitting a commutative diagram
    $$\xymatrix{
    &&&&\Vecc_{\mZ}\\
    \Rep C_2^2\ar[rr]^-{F_\lambda}\ar[rrrru]^{\Theta_\infty}\ar[rrrrd]_{\Theta^\lambda_\infty}&& \cF\ar[rru]\ar[rru]\ar[rrd] \\
    &&&&\Vecc.
    }$$
\end{lemma}
\begin{proof} Let $\cA$ be as in \Cref{subsubinfty}. One can verify that any two non-zero morphisms between indecomposable objects ({\it i.e.} the $\Omega^j\unit$) in $\cA$ compose to non-zero morphisms. 
    We may thus identify $\cA$ with $\cF$. 
\end{proof}

For $m\in\mZ_{>0}\cup\{\infty\}$, denote by $J_m^\lambda$ the tensor ideal in $\Rep C_2^2$ that is the kernel of $\Theta_m^\lambda$. Note that the kernel $J_\infty$ of $\Theta_\infty$ is the unique maximal tensor ideal, the ideal of negligible morphisms $\cN$. 
\begin{lemma}\label{IncIdeal}
    For every $\lambda\in\mP^1(\bk)$, we have a chain of strict inclusions
    $$J_1^\lambda\subset J_2^\lambda\subset J_3^\lambda\subset\cdots \subset J_\infty^\lambda\subset J_\infty=\cN.$$
\end{lemma}
\begin{proof}
Firstly, it is clear from the formulation of \Cref{ThmClass} that there are no equalities between these ideals.

    To prove that $J_m^\lambda\subset J^\lambda_\infty$, by \Cref{subsubinfty} it is sufficient to show that $J_m^\lambda$ does not contain $f_\mu$ for $\mu\not=\lambda$.
Consider therefore the short exact sequence in \Cref{seslambda} for any $\mu\not=\lambda$ and the corresponding exact triangle in $\Stab C_2^2$. By taking the tensor product with $A_{m+1}(\lambda)$, we obtain by \Cref{FusRules} an exact triangle
    $$A_{m+1}(\lambda)\;\xrightarrow{f_\mu\otimes A_{m+1}(\lambda)} A_{m+1}(\lambda)\;\to\; 0\;\to,$$
    showing that $f_\mu\otimes A_{m+1}(\lambda)$ is an isomorphism in $\Stab C_2^2$. It thus follows indeed that $\Theta_m^\lambda(f_\mu\otimes A_1(\lambda))$ is not zero.

    That $J_m^\lambda\subset J_{m+1}^\lambda$ can be seen as follows. For any $m$, the morphisms $\unit\to\Omega^j\unit$ in $J_m^\lambda$ are the same as in $J^\lambda_\infty$. All morphisms $\unit\to A_n(\mu)$ are in $J_m^\lambda$ when $\mu\not=\lambda$. Finally, the space of morphisms $\unit\to A_n(\lambda)$ in $J_m^\lambda$ is always of codimension zero or one, so this space is either everything or the subspace from \Cref{LemInc}(2). Furthermore, the codimension is zero if and only if $A_n(\lambda)$ is sent to zero (except for $n=1$ and $\lambda\not\in\mP^1(\mF_2)$, in which case it is always zero).
\end{proof}


\section{A family of \texorpdfstring{$\bO$}{O}-functors for \texorpdfstring{$p=2$}{p=2}}\label{FinSec}
In this section, we continue assuming that $\chara(\bk)=2$.  Before going into full detail, let us state the principal outcome of this section in simple terms.
\begin{theorem}
    Let $\alpha\in\bk$ be a root of the polynomial $X^2+X+1$. There is a unique LSM functor $\Rep C_2^2\to\Ver_4$ that does not send $A_1([\alpha:1])$ to zero. Taking the transitive subgroup $C_2^2<S_4$ then produces a $\Frob^2$-LSM functor
    $$\cC\xrightarrow{X\mapsto X^{\otimes 4}}\cC\boxtimes\Rep C_2^2\to\cC\boxtimes\Ver_4$$
    that commutes with all tensor functors. The functor is exact for $\cC=\Ver_4$.
\end{theorem}
This is the realisation of \Cref{Conj} for $p=2=n$, and is proved in \Cref{ThmClass} and \Cref{PropCi}. It is a plausible candidate to detect fibering over $\Ver_4$ by \Cref{ThmVer4}.

More generally, for each $\lambda\in\mP^1(\bk)\backslash\mP^1(\mF_2)$,  we use \Cref{ThmOfromV} to obtain an $\bO^2_2$-functor $\Phi^\lambda_1$ from the $\bV$-functor $\Theta_1^\lambda$ in \Cref{CorClass}.
For all other cases in \Cref{CorClass} we similarly obtain an $\bO^2\{\cV\}$-functor $\Phi^\lambda_m$ or $\Phi_\infty$, for $\cV$ some tannakian category. By taking the descendant along the fibre functor, we thus have associated $\bO^2_1$-functors.

\begin{lemma}\label{EasyImpEx}
    Let $\cC$ be a tensor category and fix $\lambda\in\mP^1(\bk)$ and $m\in\mZ_{>0}$.
    \begin{enumerate}
        \item If $\cC$ is $\Phi_\infty$-exact, it is also $\Phi_\infty^\lambda$-exact.
        \item If $\cC$ is $\Phi_\infty^\lambda$-exact, it is also $\Phi_m^\lambda$-exact.
        \item If $\cC$ is $\Phi_{m+1}^\lambda$-exact, it is also $\Phi_m^\lambda$-exact.
    \end{enumerate}
\end{lemma}
\begin{proof}
    This follows directly from \Cref{PropExIdeal}(1) and \Cref{IncIdeal}.
\end{proof}

In the rest of the section we obtain more refined results about these exactness criteria. 

\begin{remark}
    The functor $\Phi_1^\lambda$ for $\lambda\not\in\mP^1(\mF_2)$ is precisely the conjectural $\bO_2^2$-functor from \Cref{Conj} in the one case where the conjecture is proved thus far.
\end{remark}

\subsection{New characterisations of tannakian categories}

\begin{lemma}\label{LemPhiinf}
    If a tensor category is $\Phi_\infty$-exact, then it is Frobenius exact.
\end{lemma}
\begin{proof}
   In order to find a contradiction, we consider a $\Phi_\infty$-exact tensor category~$\cC$ that is not Frobenius exact. By adjunction and \Cref{Thm:FE} a tensor category is $\Phi_\infty$-exact if and only if no monomorphism $\alpha:\unit\to X$ is annihilated by $\Phi_\infty$. Using the notion of coinvariants from e.g.~\cite{Tann}, it follows, for instance using an exact faithful functor $\cC\to\Vecc$, that $\Phi_\infty(\alpha)\not=0$ if and only if
   the composite
$$\unit\xrightarrow{\alpha^{\otimes 4}} X^{\otimes 4}\tto \mathrm{H}_0(C_2^2, X^{\otimes 4}),$$
is non-zero. Here, the right-hand term represents the coinvariants with respect to the restriction of the braid action to $C_2^2<S_4$. Since $\cC$ is not Frobenius exact, we have a monomorphism $\alpha:\unit\to X $ with $\Fr(\alpha)=0$, so also $\Fr(\alpha)^{\otimes 2}=\Fr(\alpha^{\otimes 2})=0$. In other words, the composite 
$$\unit\xrightarrow{\alpha^{\otimes 4}}X^{\otimes 4}\tto \mathrm{H}_0(C_2, X^{\otimes 4})$$
is zero, where $C_2<C_2^2<S_4$ is generated by $(13)(24)$, a contradiction.
\end{proof}

\begin{proposition}\label{PropPhiinf}
    Let $\cC$ be a tensor category with {\rm one} of the following properties
    \begin{enumerate}
        \item $\cC$ has property (P);
        \item The $\Phi_\infty$-type of $\cC$ is $\Vecc\subset\Vecc_{\mZ}$.
        \item Every object in $\cC$ is algebraic.
    \end{enumerate}
    If $\cC$ is $\Phi^\lambda_\infty$-exact, for some $\lambda\in\mP^1(\bk)$, then it is also Frobenius exact and $\Phi_\infty$-exact.
\end{proposition}
\begin{proof}

    By \Cref{LemPhiinf} it is sufficient to show that $\cC$ is $\Phi_\infty$-exact.

    Part (1) follows from \Cref{IncIdeal} and \Cref{PropExIdeal}.


Denote by $\cF il\cC$ the category of $\mZ$-filtered objects in $\cC$, which we can identify with the category of $C$-comodules in the category of finite-dimensional filtered vector spaces $\cF$, if $\cC$ is equivalent to the category of $C$-comodules for a coalgebra $C$.
Hence \Cref{PropFilt} implies there exists an LSM functor $\cC\to\cF il \cC$ such that the solid arrows below yield a commutative diagram of LSM functors\begin{equation}
\label{eqFilt}\xymatrix{
    &&&&\cC\boxtimes\Vecc_{\mZ}\\
    \cC\ar[rr]\ar[rrrru]^{\Phi_\infty}\ar[rrrrd]_{\Phi^\lambda_\infty}&& \cF il\cC\ar[rru]\ar[rru]\ar[rrd] \\
    &&&&\cC.
    }
    \end{equation}
The functor $\cF il\cC\to\cC\boxtimes\Vecc_{\mZ}$ is given by taking the associated $\mZ$-graded object and the functor $\cF il\cC\to\cC$ is given by forgetting the filtration.

The above considerations give a second proof for part (1), as an object in $\cF il\cC$ is annihilated by the functor to $\cC\boxtimes\Vecc$ if and only if it is annihilated by the functor to $\cC$, so we can apply \Cref{ExFin}.

Under assumption (2), all $\mZ$-filtrations of objects in the image of $\cC\to\cF il\cC$ are trivial, and we can identify $\Phi_\infty$ and $\Phi^\lambda_\infty$, either by forgetting the $\mZ$-grading on $\Vecc_{\mZ}$ or embedding $\Vecc$ into $\Vecc_{\mZ}$. This implies that $\cC$ is also $\Phi_\infty$-exact, concluding the proof of case (2).

Finally, we can observe that the assumption in (3) implies the assumption in (2), so we can reduce to the previous case.
\end{proof}

\Cref{PropPhiinf} raises the following question.

\begin{question}
    Is the $\Phi_\infty$-type of every tensor category (of moderate growth) $\Vecc\subset\Vecc_{\mZ}$?
\end{question}

\begin{proposition}\label{PropPhiinf2}
For a tensor category $\cC$, a fixed $\lambda\in\mP^1(\bk)$ and $m\in\mZ_{>0}$, with $m>1$ in case $\lambda\in \mP^1(\mF_4)\backslash \mP^1(\mF_2)$, consider the following 5 conditions:
\begin{enumerate}
    \item[(a)] $\cC$ is tannakian;
        \item[(b)] $\cC$ is Frobenius exact;
        \item[(c)] $\cC$ is $\Phi_\infty$-exact;
        \item[(d)] $\cC$ is $\Phi_\infty^\lambda$-exact;
        \item[(e)] $\cC$ is $\Phi_m^\lambda$-exact.
    \end{enumerate}
Then we have the following implications:
\begin{enumerate}
    \item If $\cC$ is of moderate growth, then {\rm (a)}, {\rm (b)} and {\rm (c)} are equivalent.
    \item If $\cC$ satisfies (P), then {\rm (c)} and {\rm (d)} are equivalent.
    \item If $\cC$ is finite and the $\Ver_{2^\infty}$-conjecture is valid, then all {\rm (a)} -- {\rm (e)} are equivalent.
\end{enumerate}
\end{proposition}
\begin{proof}
Part (1) follows from \Cref{pTheorem}, \Cref{LemPhiinf} and \Cref{RemDefO}(3).

Part (2) follows from \Cref{EasyImpEx}(2) and \Cref{PropPhiinf}(1).

For part (3), we can use parts (1) and (2), so it suffices to observe that (d) implies (e) by \Cref{EasyImpEx}(1), and that (e) implies (a). The latter follows from \Cref{LemOCond} for $n=1$, and the fact that $\Ver_8^+$ is not $\Phi^\lambda_m$-exact by \Cref{PropCi} below.
\end{proof}


\subsection{Applications to incompressible categories}

Following \cite{BE}, we set
$$\cC_{2i}=\Ver_{2^{i+1}},\quad\mbox{and}\quad \cC_{2i+1}=\Ver_{2^{i+2}}^+,
\quad\text{ for }i\in\mN,
$$
yielding a chain $\cC_{j}\subset \cC_{j+1}$ of incompressible tensor categories.
\begin{proposition}\label{PropCi}
    Consider $j\in\mN$, $\lambda\in\mP^1(\bk)$ and $m\in\mZ_{>0}$, then $\cC_j$ is $\Phi_m^\lambda$-exact if and only if
    \begin{enumerate}
        \item $j=0$ (i.e.~$\cC_j=\Vecc$), or;
        \item $j\le 2$ (i.e.~$\cC_j\subset\Ver_4$) and $m=1$ and $\lambda\in\mP^1(\mF_4)\backslash\mP^1(\mF_2)$.
    \end{enumerate}
\end{proposition}

The crucial input for the proof is the following direct computation.
\begin{lemma}\label{CalcVer4}
Let $\alpha$ be a generator of $\mF_4$, viewed as an element in $\bk\subset\mP^1(\bk)$.
    As $C_2^2<S_4$-representations, we have
    $$\Hom_{\Ver_4}(P,V^{\otimes 4})\;\simeq\; A_1(\alpha)\oplus A_1(\alpha+1).$$
\end{lemma}

\begin{proof}[Proof of \Cref{PropCi}]
    By the chain of inclusions, for any given $(m,\lambda)$ there is an $l\in \mN$ such that $\cC_j$ is exact if and only if $j\le l$. As the simple projective object in $\cC_2$ tensor squares to the projective cover of $\unit$ in $\cC_1$ it also follows from \Cref{ExFin} that $\cC_1$ is $\Phi_m^\lambda$-exact if and only if $\cC_2$ is $\Phi_m^\lambda$-exact.

    By \Cref{CalcVer4} and \Cref{NewProp} it follows that $\cC_2=\Ver_4$ is $\Phi_m^\lambda$-exact if and only if $m=1$ and $\lambda\in\mP^1(\mF_4)\backslash\mP^1(\mF_2)$. 

    By the combination of the previous two paragraphs it is now sufficient to prove that $\cC_3$ is not $\Phi_1^\lambda$-exact for $\lambda\in\mP^1(\mF_4)\backslash\mP^1(\mF_2)$. This follows from \Cref{Lem+} and the observation that $\Phi_1^\lambda$ satisfies (O1) for $\lambda\in\mP^1(\mF_4)\backslash\mP^1(\mF_2)$, as follows from \Cref{CalcVer4} and \Cref{CorVVV}.
\end{proof}

\begin{corollary}\label{CorVer2n}
    Consider projective objects $P,Q$ in $\cC_i$ for $i>2$, then the $C_2^2$-representation $\Hom_{\cC_i}(P,Q^{\otimes 4})$ is a direct sum of non-trivial permutation modules, i.e.~a direct sum of modules $A_1(0)$, $A_1(1)$, $A_1(\infty)$ and $\bk C_2^2$.
\end{corollary}
\begin{proof}
    By \Cref{PropCi} and \Cref{NewProp}, we have 
    $$\Theta^\lambda_m(\Hom_{\cC_i}(P,Q^{\otimes 4}))=0$$
    for all $m>0$ and $\lambda\in\mP^1(\bk)$. The conclusion then follows from the definition of $\Theta^\lambda_m$ and the classification of indecomposable modules.
\end{proof}

\begin{proposition}\label{PropPhi2}
Fix $\lambda\in\mP^1(\bk)\backslash\mP^1(\mF_2)$.
    \begin{enumerate}
        \item $\overline{\Phi^\lambda_1}$ is the descendant of $\Phi_2^\lambda$ obtained from the fibre functor $\Rep\alpha_2\to\Vecc$.
        \item $\Phi^\lambda_1$ satisfies {\rm (O1)} if and only if $\lambda\in\mP^1(\mF_4)$. 
    \end{enumerate}
\end{proposition}
\begin{proof}
Part (1) follows from the classification in \Cref{ThmClass} and the observation that $\Theta_1^\lambda(A_3(\lambda))\simeq\unit^2$, so $\Theta_1^\lambda(A_3(\lambda))$ is not negligible. The latter observation can be computed directly, but also follows from \Cref{LemTech}, which shows that symmetric monoidal functors from the additive monoidal subcategory generated by $A_3(\lambda)$ in $\Stab C_2^2$ take values in tannakian subcategories, so $\Theta_1^\lambda(A_3(\lambda))\in\Vecc\subset\Ver_4$.

Part (2) was already observed in (the proof of) \Cref{PropCi}.
\end{proof}

\subsection{Conjectural characterisation of categories that fibre over \texorpdfstring{$\Ver_4$}{Ver4}}
For this section we fix a generator $\alpha\in\mF_4$, which we interpret as an element of $\bk\subset\mP^1(\bk)$, or more precisely of $\mP^1(\bk)\backslash\mP^1(\mF_2)$.

\begin{conjecture}\label{ConjMG}
    The following conditions are equivalent on a tensor category~$\cC$:
    \begin{enumerate}
        \item $\cC$ is $\Phi^\alpha_1$-exact and of moderate growth;
        \item $\cC$ admits a tensor functor to $\Ver_4$.
    \end{enumerate}
\end{conjecture}

Note that (2) implies (1) in \Cref{ConjMG}, by \Cref{PropCi}. We discuss some relations with other conjectures and statements:
\begin{theorem}\label{ThmFinal1}
    The following statements are equivalent:
    \begin{enumerate}
        \item \Cref{ConjMG} is true;
        \item $\Phi_1^\alpha$ satisfies {\rm (O2)};
        \item A tensor category of moderate growth is tannakian if and only if it is $\Phi_2^\alpha$-exact.
    \end{enumerate}
\end{theorem}
\begin{proof}
That (1) implies (2) follows from \Cref{PropPhi2}(2) and \Cref{Thm:ReverseO}(2). That (2) implies (1) follows from \Cref{PropPhi2}(2) and \Cref{ThmO}.

    That (2) and (3) are equivalent follows from \Cref{PropPhi2}(1).
\end{proof}

Finally, we show that \Cref{ConjMG}, specialised to finite tensor categories, is a consequence of \cite[Conjecture~1.4]{BEO}.

\begin{theorem}\label{ThmVer4}Assume the $\Ver_{2^\infty}$-Conjecture is valid.
     The following conditions are equivalent on a {\rm finite} tensor category $\cC$:
    \begin{enumerate}
        \item $\cC$ is $\Phi_1^\alpha$-exact;
        \item $\cC$ admits a tensor functor to $\Ver_4$;
        \item There exist projective objects $P,Q\in\cC$ so that the $C_2^2$-representation
        $\Hom(Q,P^{\otimes 4})$
        has an indecomposable summand in the list
        $$\{\Omega^{i}\unit\mid i\in\mZ\}\,\cup\, \{A_m(\alpha)\mid m\in\mZ_{>0}\}\,\cup\, \{A_m(1+\alpha)\mid m\in\mZ_{>0}\};$$
        \item There exist projective objects $P,Q\in\cC$ so that the $C_2^2$-representation
        $\Hom(Q,P^{\otimes 4})$
        has an indecomposable summand in the list
        $$\{\unit, A_1(\alpha), A_2(\alpha)\}.$$
    \end{enumerate}
\end{theorem}

\begin{proof}
By \Cref{PropPhi2}(2), equivalence of (1) and (2) is a special case of \Cref{PropO1}. In particular this means that $\cC$ is $\Phi^\alpha_1$-exact if and only if it is $\Phi^{\alpha+1}_1$-exact.

    Clearly (4) implies (3). To prove that (3) implies (1), by the previous paragraph it suffices to show that under assumption (3) $\cC$ is either $\Phi^{\alpha}_1$-exact or $\Phi^{\alpha+1}_1$-exact. If one of the modules $\Omega^i\unit$ or $A_m(\alpha)$ appears in $\Hom(Q,P^{\otimes 4})$, then $\cC$ is $\Phi_1^\alpha$-exact by \Cref{NewProp}.
    If $A_m(1+\alpha)$ appears, $\cC$ must likewise be $\Phi^{\alpha+1}_1$-exact.

    Now we prove that (2) implies (4). Assume first that $\cC$ is tannakian, then writing the forgetful tensor functor $\cC\to\Vecc$ as $\Hom_{\cC}(P,-)$ for some projective object $P$, shows that $\unit$ appears as a $C_2^2$-summand in $\Hom_{\cC}(P,Q^{\otimes 4})$, for any projective object $Q\in\cC$.

Now assume that $\cC$ is not tannakian, but has a tensor functor to $\Ver_4$. In particular, $\Phi_1^\alpha$ is faithful on $\cC$. Hence, by \Cref{NewProp}, there must be $P,Q$ as in (3) so that $\Omega^i\unit$ or $A_m(\alpha)$, for some $i\in\mZ$ or $m\in\mZ_{>0}$, appears in $\Hom_\cC(P,Q^{\o4})$. However, if $\Omega^i\unit$ or $A_{m+1}(\alpha)$ for $m>1$ appeared, it would follow, again from \Cref{NewProp}, that $\cC$ is $\Phi_m^\alpha$-exact and hence tannakian by \Cref{PropPhiinf2}(3), contradicting our assumption. Hence only $A_1(\alpha)$ and $A_2(\alpha)$ are allowed to appear in this case.
\end{proof}

\subsection*{Acknowledgements}
The authors thank Pavel Etingof for useful discussions. The research of KC was partially supported by ARC grants DP210100251 and FT220100125.
The research that led to the current paper was instigated at Max Planck Institute for Mathematics Bonn and the authors are grateful for excellent working conditions. We thank an anonymous referee for helpful comments, and for suggesting the term ``detect fibering over $\Ver_{p^n}$''.

\end{document}